\documentclass[12pt]{article}

\input epsf.tex

\usepackage{amsmath}
\usepackage{amsthm}
\usepackage{amsfonts}
\usepackage{amssymb}
\usepackage{graphicx}
\usepackage{latexsym}

\usepackage{amsmath,amsthm,amsfonts,amssymb}

\usepackage{epsfig}

\theoremstyle{plain}
\newtheorem{thm}{Theorem}[section]
\newtheorem{prop}[thm]{Proposition}
\newtheorem{lem}[thm]{Lemma}
\newtheorem{cor}[thm]{Corollary}

\theoremstyle{definition}
\newtheorem{defn}{Definition}
\theoremstyle{remark}
\newtheorem{remark}{Remark}

\advance \topmargin by -\headheight
\advance \topmargin by -\headsep
\evensidemargin \oddsidemargin
\def\sB{{\mathcal B}}
\def\tsB{{{\mathcal B}}}

\def\Block{{\textrm{Block}}}

\def\bphi{{\widetilde \psi}}
\def\tphi{{\tilde \phi}}

\def\sG{{\mathcal G}}

\def\bg{{g_0}}

\def\Gen{{\textrm{Gen}}}

\def\veta{{{\vec \eta}}}
\def\bh{{\bar h}}

\def\sH{{\mathcal H}}

\def\sL{{\mathcal L}}

\def\tM{{ M}}

\def\N{{\mathbb N}}
\def\sN{{\mathcal N}}

\def\Nb{{\textrm{nbhd}}}

\def\sP{{\mathcal P}}

\def\Par{{\textrm{Par}}}

\def\tpi{{\tilde \pi}}

\def\R{{\mathbb R}}

\def\Res{{\textrm{Res}}}

\def\sS{{\mathcal S}}

\def\support{{\textrm{support}}}

\def\chix{{\raise.5ex\hbox{$\chi$}}}

\def\Z{{\mathbb Z}}

\begin{document}
\title{Invariant measures on the space of horofunctions of a word hyperbolic group}
\author{Lewis Bowen}
\begin{abstract}
We introduce a natural equivalence relation on the space $\sH_0$ of horofunctions of a word hyperbolic group that take the value $0$ at the identity. We show that there are only finitely many ergodic measures that are invariant under this relation. This can be viewed as a discrete analog of the Bowen-Marcus theorem. Furthermore, if $\eta$ is such a measure and $G$ acts on a space $(X,\mu)$ by p.m.p. transformations then $\eta \times \mu$ is virtually ergodic with respect to a natural equivalence relation on $\sH_0\times X$. This is comparable to a special case of the Howe-Moore theorem. These results are applied to prove a new ergodic theorem for spherical averages in the case of a word hyperbolic group acting on a finite space. 

\end{abstract}
\maketitle

\noindent

\section{Introduction}
Let $G$ be a nonelementary word hyperbolic group with symmetric generating set $A$. For $g \in G$ and $n\ge 0$, let $B(g,n)$, $S(g,n)$ denote the ball and sphere of radius $n$ (in the word metric) centered at $g$ respectively. We will write $G \curvearrowright (X,\mu)$ to mean that $(X,\mu)$ is a Borel probability space on which $G$ acts by measure-preserving transformations. This paper proves the following.

\begin{thm}\label{thm:app}
There exists a finite index subgroup $G_0<G$ (depending only on $A$) such that the following holds. Suppose the action $G \curvearrowright (X,\mu)$ is ergodic and that $X$ is {\it finite} (i.e., $X$ can be identified with a finite coset space $G/H$ for some $H<G$ and $\mu$ is the uniform measure). Let $K \subset G$ be any left transversal for $G_0$ in $G$ (so $KG_0=G$ and $|K|=|G/G_0|$). Then for any function $f:X \to \R$ and for any $x \in X$,
\begin{eqnarray*}
\int f d\mu &=& \lim_{n \to \infty} \frac{1}{|K||S(e,n)|} \sum_{g \in S(e,n)}  \sum_{k \in K} f(gkx).
\end{eqnarray*}
\end{thm}
\begin{cor}\label{cor:app}
Let $G,K$ be as above. Let $({\bar G},\mu)$ be the profinite completion of $G$ (assuming $G$ is residually finite) with Haar probability measure $\mu$. If $f: {\bar G} \to \R$ is continuous and $x \in {\bar G}$ then
\begin{eqnarray*}
\int f d\mu &=& \lim_{n \to \infty} \frac{1}{|K||S(e,n)|} \sum_{g \in S(e,n)}  \sum_{k \in K} f(gkx).
\end{eqnarray*}
\end{cor}

There are specific cases in which $G_0$ cannot equal $G$. For example, if $G$ is a finitely generated nonabelian free group and $A$ is a free generating set, then consider the action of $G$ on $\Z/2Z$ induced by the homomorphism $G \to \Z/2\Z$, $a \to 1$ for all $a \in A$. $G_0$ cannot be chosen to equal $G$ for this action; consider $f$ to be the indicator function of the set $\{0\}$ to see that the above limit would not converge.

This theorem is implied by the following stronger statement. Let $\Z^G$ be the space of all functions $h:G \to \Z$ with the topology of uniform convergence on finite subsets. $G$ acts on $\Z^G$ in the usual way: $gh:G\to \Z$ is defined by $gh(f)=h(g^{-1}f)$ (for $h\in \Z^G, f, g \in G$). For each $n>0$ define $h_n:G \to \Z$ by $h_{n}(g) = d(g,e)-n$ where $d(\cdot,\cdot)$ denotes distance in the word metric. For $x \in X$, let $u_{n,x}$ be the uniform measure on the collection $\{(gh_n,gx) |\, g\in S(e,n)\}$. We extend this measure to all of $\Z^G \times X$ by setting $u_{n,x}(E)=0$ for all sets $E$ in the complement of this collection.

In general, for a topological space $Z$, let $M(Z)$ denote the space of all Borel probability measures on $Z$ with the weak* topology. Recall this means that a sequence of measures $\{\omega_n\}$ converges to $\omega$ iff for every continuous function $f:Z \to \R$,  $\int f d\omega_n$ converges to $\int f d\omega$. If $Z$ is compact and metrizable then the Banach-Alaoglu theorem implies that $M(Z)$ is compact. Note $u_{n,x} \in M(\Z^G \times X)$ where $X$ has the discrete topology.

\begin{thm}\label{thm:app2}
Let $G_0, K$ be as in the previous theorem. If $G \curvearrowright (X,\mu)$ is ergodic and $X$ is finite then for any $x \in X$, every subsequential weak* limit point of the sequence
$$ \Big\{\frac{1}{|K|}\sum_{k \in K} u_{n,kx}\Big\}$$ is of the form $\eta \times \mu$ for some probability measure $\eta \in M(\Z^G)$.
\end{thm}



Theorem \ref{thm:app2} immediately implies the former result. The focus of this paper is on the set of possibilities for $\eta$ and the ergodic decomposition of $\eta \times \mu$ (appropriately defined) in the general case (i.e., $X$ is {\it not} assumed to be finite). To begin, let us consider what the support of $\eta$ could be. It is necessarily contained in the set of all possible limits of sequences of the form $\{g_nh_n\}$ where $g_n \in S(e,n)$. To describe these limits we need more notation.

Let $\Gamma=(G,A)$ be the Cayley graph of $G$. We regard it as a path-metric space by declaring that each edge is isometric to the unit interval. If $h:G \to \Z$ is any function, then $h$ may be extended to all of $\Gamma$ by defining $h(x)=th(v)+(1-t)h(w)$ whenever $x$ is the point on the edge from $v$ to $w$ ($v,w \in G$) such that $d(x,v)=t$. We will not distinguish between $h$ and its extension to $\Gamma$.

If $g_n \in S(e,n)$ and $h_n$ is defined as above then it can be shown (following [CP01, proposition 2.9] ) that every subsequential limit point $h_\infty:G \to \Z$ of the sequence $g_nh_n$ satisfies the following two conditions:
\begin{itemize}
\item $h_\infty$ is {\bf $\epsilon$-convex}, i.e., for all geodesic segments $[x_0,x_1] \subset \Gamma$ and for every $t\in [0,1]$
$$h_\infty(x_t)\le t h_\infty(x_0) + (1-t)h_\infty(x_1)+\epsilon,$$
where $x_t$ is the point on $[x_0,x_1]$ satisfying $|x_0-x_t|=t|x_0-x_1|$ and $\epsilon$ is some positive number. 
\item $h_\infty$ is {\bf distance-like}, i.e., for every $x\in \Gamma$ and every $\lambda \in \R$ with $h_\infty(x)\ge \lambda$
$$h_\infty(x)=\lambda+d(x,h_\infty^{-1}(\lambda)).$$
\end{itemize}
In general, a function $h: \Gamma \to \R$ is an {\bf $\epsilon$-horofunction} on $\Gamma$ if it is $\epsilon$-convex and distance-like. In [CP01, Corollary 4.8] it is proven that any $\epsilon$-horofunction is a $68\delta$-horofunction where $\delta$ is the hyperbolicity constant of $\Gamma$. Let $\sH$ denote the space of all horofunctions with range in the integers. Let $\sH_0 \subset \sH$ denote the compact subspace of horofunctions $h$ satisfying $h(e)=0$. 

 The measure $\eta$ in the above theorem is necessarily supported on $\sH_0$. $\eta$ also has important symmetry properties. To describe these, we recall some definitions from the theory of measured equivalence relations.

Let $Y$ be a Borel space. A Borel equivalence relation $R \subset Y \times Y$ is {\bf discrete} if each of its equivalence classes is countable or finite. A {\bf partial transformation} of $R$ is a Borel bijection $\phi:$ Dom $\phi \to$ Im $\phi $ whose graph is contained in $R$. A measure $\eta$ on $Y$ is {\bf $R$-invariant} if for any partial transformation $\phi$, $\phi_* \eta = \eta$.  We will denote by $M(Y)$ the space of all Borel probability measures on $Y$ and by $M_R(Y)$ the space of all $R$-invariant Borel probability measures.

Given a set $S \subset Y$, the {\bf $R$-saturation} $[S]$ is defined by $[S]:= \{y \in Y | \, (y,s)\in R $ for some $s \in S\}$. $S$ is {\bf $R$-saturated} if $[S]=S$. A measure $\eta$ on $Y$ is {\bf ergodic} if for every $R$-saturated set $S \subset Y$, either $\eta(S)=0$ or $\eta(Y-S)=0$. 

If $G$ is a group acting on $Y$ then the {\bf induced equivalence relation} $R$ on $Y$ is $R=\{(y, gy) \in Y \times Y|\, y \in Y, g \in G\}$. If $Z \subset Y$ then the {\bf restriction of $R$ to $Z$} is the equivalence relation on $Z$ equal to $R \cap Z \times Z$.

Now, the action of $G$ on $\Z^G$ induces an equivalence relation on $\Z^G$ and by restriction, an equivalence relation $R$ on $\sH_0$. The measure $\eta$ in theorem \ref{thm:app2} is in $M_R(\sH_0)$. 

\begin{thm}\label{thm1}
$M_R(\sH_0)$ is nonempty and there are only finitely many ergodic measures in $M_R(\sH_0)$. 
\end{thm}

This is proven in section 4. For example, if $G$ is a finitely generated free group and $A$ is a free generating set, it can be shown that there is only one $R$-invariant probability measure on $\sH_0$. Indeed, $\sH_0$ can be identified with the boundary $\partial \Gamma$ by the map that associates to $h \in \sH_0$, the unique ``point at infinity'' that equals the limit set of the horosphere $\{g \in G|\, h(g)=0\}$. The unique $R$-invariant probability measure is the Patterson-Sullivan measure on the boundary. I do not know of a single example in which $M_R(\sH_0)$ contains more than one measure.

Suppose now that $G \curvearrowright (X,\mu)$. We do {\it not} assume that $X$ is finite. $G$ acts on $\Z^G \times X$ diagonally. This induces an equivalence relation on $\Z^G \times X$ and by restriction, an equivalence relation on $\sH_0 \times X$. Let $M_R(\sH_0 \times X)$ denote the space of Borel probability measures that are invariant under this relation. The main result of this paper is:

\begin{thm}\label{thm:main}
If $\eta \in M_R(\sH_0)$  and $G$ acts ergodically on $(X,\mu)$ then there exists ergodic measures $\omega_1,...,\omega_q \in M_R(\sH_0 \times X)$ and real numbers $t_i \ge 0$ such that
$$\eta \times \mu = t_1\omega_1 + ... + t_q \omega_q.$$
Moreover, the number $q$ of ergodic components is bounded by a constant $Q$ that depends only on $(G,A)$ and not on $(X,\mu)$.
 \end{thm}
By making small modifications to the arguments in this paper, it can be shown that if $G$ is a nonabelian free group and $A$ is a free generating set then the number $q$ of ergodic components in the theorem above is at most equal to $2$. In fact, the action of $G$ on $X=\Z/2\Z$ obtained from the homomorphism $G \to \Z/2\Z$ defined by $a \to 1$ for all $a\in A$, requires that $q=2$. 

 It is interesting to compare this result with the Howe-Moore theorem: if $\sG$ is a semisimple Lie group with finite center, $\sG \curvearrowright (X,\mu)$ and the restriction of this action to any simple noncompact factor of $\sG$ is ergodic then for every subgroup $H<G$ such that $H$ has noncompact closure in $G$, the action of $H$ on $X$ is strongly mixing. In particular, the induced action of any unipotent subgroup is ergodic.

 It is easy to construct ergodic actions of a free group $G$ such that for some infinite subgroup $H<G$, the induced action of $H$ is nonergodic. Hence the straightforward analogue of the Howe-Moore theorem for arbitrary word hyperbolic groups fails. This is not surprising since free groups are far from being semisimple.

In the conclusion section of this paper, we describe a more general framework from which to view these results.

\subsection{History}

Theorem \ref{thm1} can be regarded as a discrete analog of the Bowen-Marcus theorem: there is a unique holonomy-invariant transverse probability measure on the strong unstable foliation of the geodesic flow on a compact manifold with pinched negative curvature [BM77]. Their proof shows that this measure is induced from the well-known Bowen-Margulis measure on the unit tangent bundle which is the measure of maximal entropy of the geodesic flow. An analogue of the geodesic flow for word hyperbolic groups was defined by Gromov and developed by Coornaert and Papadopoulos [CP02]. The present work builds on the related paper [CP01].


Theorem \ref{thm:app} may be regarded as a pointwise ergodic theorem, a mean ergodic theorem or an equidistribution theorem because the three notions coincide when the action space $X$ is finite. The mean ergodic theorem for a free group with respect to a free generating set was first proven by Guivarc'h [Gu69]. Pointwise ergodic theorems for ball and spherical averages for the free group with respect to a free generating set were first proven in [NS94] for all $L^p$ functions with $p>1$. Bufetov gave a very elegant proof which extends to $L\log L$ functions [Bu02] and to all Markov groups satisfying a certain symmetry condition. The only pointwise ergodic theorem in the literature for arbitrary word hyperbolic groups is in [FN98]. There it is proven that if the action of $G$ on $(X,\mu)$ is exponentially mixing then the Cesaro averages of spherical averages of an $L^p$ function ($p>1$) converge pointwise a.e. to the space average.

 In general, there are very few mean or pointwise ergodic theorems known for ball or spherical averages with respect to the {\it word} metric of a {\it discrete} nonamenable group. For example, there are no known mean or pointwise ergodic theorems for ball or spherical averages in the case of the free group $G$ with respect to an {\it arbitrary} symmetric generating set $A$. The continuous case is better understood. Pointwise ergodic theorems for ball averages with respect to a word metric on a connected simple Lie group with finite center are proven in [Ne07] (see also [GN07]). There is an excellent survey article [Ne06] where these theorems (and many other related results) are discussed.

{\bf Acknowledgments}: I would like to thank Amos Nevo for inspiring conversations about pointwise ergodic theorems. I would also like to thank Russ Lyons for introducing me to unimodular networks. These objects are, roughly speaking, another way of formulating graphed measured equivalence relations. I was inspired by [AL07] to think of an analogue of the maximal unipotent subgroup of $SO(n,1)$ for word hyperbolic groups. I'd also like to thank Chris Connell and Lorenzo Sadun for helpful conversations. And I'd like to thank Alex Furman for directing my attention to [Ka03] which is used here in a crucial way. 

\section{Organization}
In \S \ref{sec:prelim} we recall standard definitions regarding word hyperbolic groups. In \S \ref{sec:thm1} we prove theorem \ref{thm1}. That section is outlined separately below. The tools developed in \S \ref{sec:thm1} (especially \S \ref{symbolic coding}) are used in \S \ref{sec:part2} and \S \ref{sec:part1} to prove theorem \ref{thm:main}. There is an important shift operator on $M_R(\sH_0)$ defined in \S \ref{sec:periodicity}. In \S \ref{sec:part1} we prove that every ergodic component of $\eta \times \mu$ is virtually invariant under a related shift operator. The proof relies on a key lemma that is proven separately in \S \ref{section:horospherical}. In \S \ref{sec:part2} we prove of theorem \ref{thm:main} assuming the results of \S \ref{sec:part1}. In \S \ref{sec:app} we prove theorem \ref{thm:app2} and corollary \ref{cor:app}. In the conclusion \S \ref{sec:concl} we present some of the intuitive notions and speculations that led to this paper.

\S \ref{sec:thm1} is the longest section of this paper. We first prove that $M_R(\sH_0)$ is nonempty. In the \S \ref{symbolic coding}, we recall a symbolic coding of $\sH_0$ introduced in [CP01]. This coding is in terms of `blocks'. In \S \ref{mtp}, we discuss the mass-transport principle, which is a tool for computing the values of an $R$-invariant measure. In \S \ref{s1}, we show that any measure $\eta \in M_R(\sH_0)$ is determined by its block densities. In \S \ref{generation}, we use the Patterson-Sullivan theory developed in [Co93] to show that `generation growth' is the roughly the same as the growth of the group. This is used in \S \ref{sec:periodicity}, together with the theory of nonnegative matrices, to conclude that the block densities of a measure $\eta \in M_R(\sH_0)$ form a sequence of eigenvectors of a certain nonnegative matrix. This is then used to show that $\eta$ is virtually invariant under a natural shift-operator. In \S \ref{sub:alpha} this is used to show that $M_R(\sH_0)$ is isomorphic to $M_R(\sH_*)$ where $\sH_*=\{h \in \sH\, | \, h(e)\le 0\}$. In \S \ref{sub:qc}, we show that if $\eta$ is invariant under the shift-operator then its projection to the boundary is quasiconformal. We then use the fact (proven in [Co93]) that quasiconformal measures on the boundary are equivalent to conclude theorem \ref{thm1}.


\section{Word Hyperbolic Groups}\label{sec:prelim}


A detailed discussion of the notion of $\delta$-hyperbolicity and of the associated structures can be found in the seminal work of [Gr87] and in the notes [GdlH90]. Below are listed some of the definitions and properties used later on.

We shall choose the definition based on the {\it Rips condition}: a non-compact complete proper geodesic metric space $\Gamma$ is {\bf $\delta$-hyperbolic} (with $\delta \ge 0$) if each of the sides of any geodesic triangle is contained in the $\delta$-neighborhood of the union of the other two sides (see, for instance, [GdlH90, Proposition 2.21] for a list of other equivalent definitions). The minimal number $\delta$ with this property is the {\bf hyperbolicity constant} of $\Gamma$. A graph is called $\delta$-hyperbolic if the associated $1$-complex with length $1$ edges is $\delta$-hyperbolic. Usually, we shall not be concerned with the precise value of the hyperbolicity constant, and call the above spaces just {\bf hyperbolic}. A discrete group is {\bf word hyperbolic} if one (and hence, any) of its Cayley graphs is hyperbolic.

\subsection{The Hyperbolic Boundary}
Fix a hyperbolic space $\Gamma$ with metric $d$ and hyperbolicity constant $\delta$. Denote by
$$(y|z)_x = \frac{1}{2}[d(x,y)+d(x,z)-d(y,z)], \, x,y,z \in \Gamma$$
the {\bf Gromov product} on $\Gamma$. If $(x_n)$ is a sequence of points in $\Gamma$, then $(x_n)$ {\bf converges at infinity} if $(x_p|x_q)_{x_0}\to\infty$ as $p,q \to \infty$. This does not depend on the choice of $x_0$. Two sequences $(x_n),(y_n)$ are {\bf equivalent} if $(x_n|y_n)_{x_0} \to \infty$ as $n \to \infty$. The boundary of $\Gamma$, denoted by $\partial \Gamma$, is the set of equivalence classes of sequences $(x_n)$ that converge at infinity.

If $\xi \in \partial \Gamma$, then we say that $(x_n)$ converges to $\xi$ if $\xi$ is the equivalence class of $(x_n)$. It is well-known that if $r:[0,\infty)\to \Gamma$ is a geodesic ray, then for every sequence $\{t_n\}$ with $t_n\to\infty$, $r(t_n)$ converges at infinity to some point $\xi$ that depends only on $r$.

\section{$R$-invariant measures}\label{sec:thm1}
From here on, let $G$ be a fixed word hyperbolic group with finite symmetric generating set $A$. Let $\delta$ be an integer that is greater than the hyperbolicity constant of $\Gamma$, the Cayley graph of $G$ with respect to $A$. Here we will prove that $M_R(\sH_0)$ is nonempty. This result and its proof are not used again until section \ref{sec:app}.

\begin{lem}\label{lem:existence}
$M_R(\sH_0)$ is nonempty.
\end{lem}

\begin{proof}
Consider the space $\Z^G$ of all functions $h: G\to \Z$ with the uniform topology on compact sets. This space is metrizable but noncompact. So we first identify a nice compact subspace. Let $\Z^G_0 \subset \Z^G$ be the space of functions $h: G\to \Z$ satisfying
\begin{itemize}
\item $h(e)=0$,
\item for all $g \in G$, $|h(g)| \le d(g,e)$.
\end{itemize}
The subspace $\Z^G_0$ is compact and $\sH_0 \subset \Z^G_0$. 

Recall from the introduction the following. $G$ acts on $\Z^G$ in the usual way: $gh(f)=h(g^{-1}f) ~\forall h\in \Z^G, g,f \in G$. This action induces an equivalence relation on $\Z^G$ which restricts to an equivalence relation on $\Z^G_0$. Let $M_R(\Z^G_0)$ denote the space of all Borel probability measures on $\Z^G_0$ that are invariant under all partial transformations of this relation. Because $\Z^G_0$ is compact, $M_R(\Z^G_0)$ is weak* compact.

For each $n>0$ let $h_{n}:G \to \Z$ be defined by $h_{n}(f) = d(f,e)-n$. Observe that if $g \in S(e,n)$, then $gh_n(e)=h_n(g^{-1})=d(g^{-1},e)-n=0$. So $gh_n \in \Z^G_0$. Conversely, if, for some $g \in G$, $gh_n \in \Z^G_0$ then $d(g,e)=d(g^{-1},e)=n$. 

Let $u_n$ be the uniform measure on the collection $\{g h_n |\, g\in S(n)\}$. The above discussion implies that $u_n \in M_R(\Z^G_0)$. By weak* compactness, the sequence $\{u_n\}$ has a subsequential limit point $u_\infty \in M_R(\Z^G_0)$. We claim that $u_\infty \in M_R(\sH_0)$. It suffices to prove that if $\{g_n\}$ is any sequence with $g_n \in S(e,n)$ then every subsequential limit point of the sequence $\{g_nh_n\}$ is an element of $\sH_0$. The proof of this fact is almost identical to the proof in [CP01, proposition 2.9] that every Busemann function is, in fact, a horofunction. We leave the details to the reader.
\end{proof}

\subsection{A symbolic coding of the space of horofunctions}\label{symbolic coding}
In [CP01], an explicit homeomorphism of $\sH_0$ onto a subshift of finite type over the natural numbers was constructed using blocks (which will be defined in this section). Our notation differs from [CP01].

From here on, fix a total ordering of the generating set $A$.

\begin{defn}
For $h\in \sH$ and $g\in G$ let $\Par_h(g) = ga \in G$ where $a \in A$ is the least element of $A$ satisfying $h(ga)=h(g)-1$. Such an element exists by the distance-like property of horofunctions. $\Par_h(g)$ is the {\bf parent} of $g$ with respect to $h$.
\end{defn}

\begin{defn}
Define $\Par:\sH\to \sH$ by $\Par(h)=\Par_h(e)^{-1}h$. $\Par(h)$ is the {\bf parent} of $h$.
\end{defn}

\begin{defn}
Fix integers $H,W>0$. For $h \in \sH$, let
$$\Block(h) = \big\{g \in G|  \, \exists n \in [0, H]\, \textrm{ s.t. } d(\Par^n_h(e),g) \le W \textrm{ and } h(g)=h(\Par^n_h(e))\big\}.$$
We will impose restrictions on the constants $H,W$ after theorem \ref{thm:CPmain} below. 
\end{defn}

\begin{defn}
Let $R_B \subset \sH \times \sH$ be the equivalence relation $(h_1,h_2) \in R_B$ if 
\begin{itemize}
\item $\Block(h_1)=\Block(h_2)$ and
\item there is a constant $C$ such that $h_1(g)=h_2(g)+C$ for all $g \in \Block(h_1)$.
\end{itemize}

\end{defn}

\begin{defn}\label{defn:tsB}
Let $\tsB$ be the set of all $R_B$-equivalence classes (called {\bf blocks}). It is a finite set. By abuse of notation we use $\Block(h)$ to denote the $R_B$-equivalence class of $h$ in $\sH$.

We identify $\tsB$ with the vertex set of a directed graph, also denoted by $\tsB$, as follows. If $a \in A$, $h \in C \in \tsB$, $ah \in B \in \tsB$ and $\Par(h)=ah$ then there is an edge from $B$ to $C$. Observe that if $h \in C$ then there is an $a \in A$ that depends only on $C$ such that $\Par(h)=ah$. Therefore, there is at most one edge from $C$ to $B$. Thus, $\tsB$ does not contain multiple edges.
\end{defn}

For $h \in \sH_0$, let $P(h): \N \to \tsB$ be the reverse-directed path in $\tsB$ given by 
$$P(h)(n)=\Block(\Par^n(h)).$$
Let $\sP$ be the set of all reverse-directed paths $p:\N \to \tsB$. It carries the topology of uniform convergence on finite sets. So it is homeomorphic to a Cantor set.

\begin{thm}[CP01, theorem 8.18]\label{thm:CPmain}
If $W>0$ is sufficiently large and $H>0$ is sufficiently large (how large depends on $W$) then the map $P:\sH_0 \to \sP$ is a homeomorphism.
\end{thm}

In [CP01], explicit bounds for $H,W$ are given that depend only on the hyperbolicity constant $\delta$. Now fix constants $H,W$ so that $H > W + 16\delta > 32\delta+100$ and such that the above theorem is true for $H,W$. We will show that if we know the block type of some $h \in \sH$, then we know all of the block types of the ``children'' of $h$. This will require some ideas from [CP01] which we recall next. 
\begin{defn}
 If $h \in \sH$ and $r:I \subset \R \to \Gamma$ is a path parametrized by arclength such that $h(r(t))-h(r(t'))= t'-t$ for every $t, t' \in I$, then $r$ is called an {\bf $h$-gradient arc}. If $I=[0,\infty)$ then $r$ is called an {\bf $h$-gradient ray}. Because of the distance-like property of horofunctions, for every $h \in \sH$ and every $g \in G$ there exists an $h$-gradient ray with $h(0)=g$.
\end{defn}

Proposition 3.3 of [CP01] implies that $h$-gradient arcs are geodesics. Thus, if $r$ is a $h$-gradient ray, there exists a unique point $r(\infty)$ on the boundary at infinity such that $r(t)\to r(\infty)$ (as $t\to\infty$) in the natural topology on $\Gamma \cup \partial \Gamma$.

\begin{defn}
 Proposition 4.1 of [CP01] implies that for any two $h$-gradient rays $r_1,r_2$, $r_1(\infty)=r_2(\infty)$. Therefore, we may define $\pi(h)=r(\infty)$ for any $h$-gradient $r$. It is called the {\bf point at infinity for $h$}. 
\end{defn}

The next result is proposition 4.4 of [CP01] adapted to the notation here.

\begin{prop}[CP01, proposition 4.4] \label{prop4.4}
Let $h \in \sH$ and $n\ge 0$. Let $r:[-n,\infty) \to \Gamma$ be a geodesic ray such that $r(\infty)=\pi(h)$. For $t \ge -n$, let
$$R_t=\big\{g \in G \, : \, h(g)=h(r(t))\big\} \cap B\big(r(t),16\delta\big)$$
where $\delta$ is the hyperbolicity constant of $(G,A)$. Then, 
$$h(g)-h(r(t)) = d(g,R_t)$$
for all $g \in G$ and for all $t$ with $t>d(g,r(-n)) -n + 16\delta$. 
\end{prop}

\begin{lem}\label{lem:determined} 
Suppose $h_1,h_2 \in \sH$, $h_1(e)=h_2(e)$ and $\Block(h_1)=\Block(h_2)$. Let $r:(-\infty,\infty)\to\Gamma$ be an $h_1$-gradient line with $r(0)=e$. Then for every $g \in G$ and $n\ge 0$, if $d(g,r(-n)) \le n$ then $h_1(g)=h_2(g)$. In particular, $r$ restricted to $(-\infty,H]$ is an $h_2$-gradient and $\Block(r(-n)^{-1}h_1)=\Block(r(-n)^{-1}h_2)$ for all $n\ge 0$. Thus the block type of $h_1$ determines the block type of all of its ``children''.
\end{lem}

\begin{proof}
Because $\Block(h_1)=\Block(h_2)$ and $H\ge 16\delta +2$, $h_1(r(16\delta+2))=h_2(r(16\delta+2))$. For $i=1,2$ let 
$$R^i = \big\{f \in G|\, d(f,r(16\delta+ 2)) \le 16\delta\textrm{ and } h_i(f)=h_i(r(16\delta+2))\big\}.$$
Since $\Block(h_1)=\Block(h_2)$, $R^1=R^2$. 

Let $n\ge 0$. Let $g \in G$ be such that $d(r(-n),g)\le n$. The previous proposition implies
$$h_1(g)=h_1(r(16\delta+ 1)) + d(g,R^1) = h_2(r(16\delta+ 1)) + d(g,R^2) = h_2(g).$$
In particular, if $g=r(-n)$ then $h_2(g)=h_1(g)$. This shows that $r$ restricted to $[-n,H]$ is an $h_2$-gradient. Since $n$ is arbitrary this completes the proof.
\end{proof}

\subsection{The Mass Transport Principle}\label{mtp}

\begin{prop}[The Mass Transport Principle]\label{prop:mtp}
Suppose $(Y,\mu)$ is a Borel probability space, $R \subset Y\times Y$ is a discrete Borel equivalence relation and $\mu$ is $R$-invariant. Let $F:R \to \R$ be any Borel map. Then
$$\int \Big(\sum_{y_2} F(y_1,y_2) \Big) d\mu(y_1) = \int \Big(\sum_{y_1} F(y_1,y_2) \Big) d\mu(y_2).$$
\end{prop}
This principle was introduced in [Ha97] and developed further in [BLPS99].
\begin{proof}
It suffices to prove the proposition in the special case in which $F$ is the characteristic function of a Borel set $E \subset R$. Since $R$ is a discrete equivalence relation, there exists an at most countable collection $\{\phi_i\}_{i=1}^N$ of partial transformations such that $E$ is the disjoint union of the graphs of the $\phi_i$. Here $N$ is allowed to equal $\infty$. This follows, for example, from the Feldman-Moore theorem [FM77] that every discrete Borel equivalence relation is generated by the action of a countable group. 

It now suffices to prove the result in the special case in which $F=\chi_E$ and $E$ is the graph of $\phi$, a partial transformation of $R$. In this case, the left hand side of the above equation equals $\mu($dom $\phi)$ and the right hand side equals $\mu($rng $\phi)$. Since $\mu$ is $R$-invariant, $\mu($dom $\phi)=  \mu($rng $\phi)$.
\end{proof}

\begin{cor}\label{cor:mtp}
If $f:Y \to Y$ is a finite-to-1 Borel map whose graph is contained in $R$ and $E \subset Y$ is Borel then
$$\mu(E) = \int |f^{-1}(y) \cap E| d\mu(y).$$ 
\end{cor}
\begin{proof}
Apply the mass transport principle to the function $F$ defined by $F(y_1,y_2)=1$ if $f(y_1)=y_2$ and $y_1 \in E$; $F(y_1,y_2)=0$ otherwise.
\end{proof}

\subsection{Every invariant measure is determined by its block densities}\label{s1}

 For $k \in \Z$, let $\sH_k = \{h \in \sH|\, h(e)=-k\}$. For $I \subset \Z$, let $\sH_I = \{h \in \sH | -h(e) \in I\}$. For any $I \subset \Z$, the equivalence relation on $\sH$ induced by the action of $G$ restricts to an equivalence relation on $\sH_I$. Let $M_R(\sH_I)$ denote the space of relation-invariant Borel probability measures on $\sH_I$. 

For every finite subset $I\subset \Z$  containing $0$, there is a natural {\bf restriction map} $\Res:M_R(\sH_I)\to M_R(\sH_0)$ obtained by restricting $\eta \in M_R(\sH_I)$ to $\sH_0$ and normalizing.

\begin{lem}\label{restriction}
$\Res$ is an isomorphism.
\end{lem}
\begin{proof}
Let $f: \sH_I \to \sH_0$ be a uniformly finite-to-1 Borel map whose graph is contained in the relation on $\sH$. For example, there is a constant $C>0$ such that for any $h \in \sH_I$ there is some $g_h \in B(e,C)$ such that $g_hh \in \sH_0$. We could define $f(h)=g_hh$ for some Borel choice of $g_h$.

If $\omega \in  M_R(\sH_I)$ and $E$ is a Borel subset of $\sH_I $ then corollary \ref{cor:mtp} implies
$$\omega(E) = \int |f^{-1}(h) \cap E| d\omega(h).$$
So, given $\eta \in M_R(\sH_0)$ define a measure $\Psi\eta$ on $\sH_{[-k,k]}$ by
$$\Psi\eta(E) = C^{-1} \int |f^{-1}(h) \cap E| d\eta(h)$$
where $C = \int |f^{-1}(h)| d\eta(h)$ is finite and positive. $\Psi$ is the inverse of $\Res$. 
\end{proof}

\begin{defn}
For $\eta \in M_R(\sH_0)$ and $k>0$, let $\omega_k \in M_R(\sH_{[0,k]})$ be the measure whose normalized restriction to $\sH_0$ equals $\eta$. Let ${\bar \eta}_k \in M_R(\sH_k)$ be the normalized restriction of $\omega_k$ to $\sH_k$. Define
$$\eta_k = \frac{\omega_k(\sH_k)}{\omega_k(\sH_0)}  {\bar \eta}_k.$$
Note that if $l>k$ then $\frac{\omega_l(\sH_k)}{\omega_l(\sH_0)}  = \frac{\omega_k(\sH_k)}{\omega_k(\sH_0)} $. 
\end{defn}

\begin{lem}\label{lem:etak}
For any $\eta \in M_R(\sH_0)$, any $j, k\ge 0$ and any Borel $E \subset \sH_j$, 
$$\eta_j(E) = \int |\Par^{-k}(h)\cap E| d\eta_{j+k}(h).$$
\end{lem}

\begin{proof}
Let $\omega \in M_R(\sH_{[0,j+k]})$ be such that the normalized restriction of $\omega$ to $\sH_0$ is $\eta$. Since $\omega$ is relation invariant, it follows from corollary \ref{cor:mtp} that 
$$\omega(E) = \int |\Par^{-k}(h)\cap E| d\omega(h).$$
Since $E \subset \sH_j$,
$$\omega(E) = {\bar \eta}_j(E)\omega(\sH_j) = \eta_j(E)\omega(\sH_0).$$
The integrand equals zero unless $h\in \sH_{j+k}$. Hence 
\begin{eqnarray*}
\eta_j(E) &=& \frac{\omega(E)}{\omega(\sH_0)}\\
&=& \frac{1}{\omega(\sH_0)} \int |\Par^{-k}(h)\cap E| d\omega(h)\\
 &=& \int |\Par^{-k}(h)\cap E| \frac{\omega(\sH_{j+k})}{\omega(\sH_0)} d{\bar \eta}_{j+k}(h)\\
&=& \int |\Par^{-k}(h)\cap E| d\eta_{j+k}(h).
\end{eqnarray*}
\end{proof}

\begin{defn}\label{defn:veta}
  For $\eta \in M_R(\sH_0)$ and $n \in \N$, let $\veta_n$ denote the $\tsB \times 1$ vector with $B$-entry equal to $\eta_n(B)$. 
\end{defn}


\begin{prop}\label{prop:determinism}
Every $\eta \in M_R(\sH_0)$ is determined by the vector sequence $\{\veta_n\}$. I.e., if $\eta, \omega \in M_R(\sH_0)$ and $\veta_n={\vec \omega}_n$ for all $n \ge 0$ then $\eta = \omega$.
\end{prop}

\begin{proof}
For $h \in \sH$ and $F \subset G$, let $h|_F$ denote the restriction of $h$ to $F$. A {\bf cylinder set} $L \subset \sH$ is of the form $L=Cyl(h,F)$ where $F \subset G$ is finite and
$$Cyl(h,F):=\{h'\in\sH\,:\,h'|_F = h|_F\}.$$
Because cylinder sets generate the $\sigma$-algebra of Borel sets, it suffices to show that $\eta(L)=\omega(L)$ for every cylinder set $L$. Fix such a set $L \subset \sH_0$. Lemma \ref{lem:etak} implies
$$\eta(L) = \int |\Par^{-k}(h) \cap L| d\eta_k(h)$$
for any $k\ge 0$. So, fix $k > 16\delta +  \max_{f \in F} d(f,e)$. 

Fix $B \in \sB$. We claim that if $h_1, h_2 \in B \cap\sH_k$ then $|\Par^{-k}(h_1) \cap L|=|\Par^{-k}(h_2) \cap L|$. To see this, suppose $g \in G$ is such that $\Par^k(gh_1)=h_1$. Thus if $r:[0,\infty)\to\Gamma$ is the minimal $h_1$-gradient with $r(0)=g^{-1}$ then $r(k)=e$. By lemma \ref{lem:determined}, $r$ restricted to $[0,k]$ is also a minimal $h_2$-gradient. Thus $\Par^k(gh_2)=h_2$. By the same lemma, $h_1(f)=h_2(f)$ for all $f\in F$. Thus $gh_1 \in L$ iff $gh_2 \in L$. Since $g$ is arbitrary, this implies the claim.

 So we may let $N_k(L,B):=|\Par^{-k}(h)\cap L|$ for any $h \in \sH_k \cap B$. Therefore,
$$\eta(L) = \sum_{B \in \sB} N_k(L,B) \eta_k(B).$$
Since the same is true with $\omega$ replacing $\eta$, the proposition follows.
\end{proof}

\subsection{Generation growth}\label{generation}
Let
$$e(\Gamma)=\limsup_{n \to\infty} \frac{1}{n}\ln \big|S(e,n)\big|.$$
For $h \in \sH$, $S \subset G$ and $n \ge 0$, let
$$\Gen_n(h,S)=\bigcup_{g \in S} \Par_h^{-n}(g)$$
be the {\bf $n$-th generation} of the elements of $S$. The goal of this section is to prove
\begin{prop}\label{prop:generation}
There exist constants $C_1,C_2>0$ such that for any $h\in \sH_0$ if $S=B(e,C_1) \cap \{g\in G |\, h(g)=0\}$ then 
$$C_2^{-1} e^{e(\Gamma)n} \le \big|\Gen_n(h,S)\big| \le C_2 e^{e(\Gamma)n}$$
for all $n\ge 0$.
\end{prop}
The proof involves Patterson-Sullivan theory by way of [Co93].
\begin{thm}[Co93, th\'eor\`eme 7.2]\label{thm:growth}
There exists a constant $C\ge 1$ such that for all $n \ge 0$, 
$$C^{-1}\exp(e(\Gamma)n) \le |S(e,n)| \le  C\exp(e(\Gamma)n).$$
\end{thm}
 
We need the concept of a quasiconformal measure on $\partial \Gamma$. See [Co93] for more details. Fix a constant $a>1$. For $\xi \in \partial \Gamma$, let $h_\xi$ be any horofunction with point at infinity equal to $\xi$. 

\begin{defn}\label{def:qc}
Let $D\ge 0$. Let $\eta$ be a measure on $\partial \Gamma$ with $0 < \eta(\partial \Gamma) < \infty$. Then $\eta$ is {\bf $G$-quasiconformal of dimension $D$} if it is $G$-quasiinvariant and $\exists C \ge 1$ such that
$$C^{-1}a^{D(h_\xi(e) - h_\xi(g))} \le \frac{d(g_*\eta)}{d\eta}(\xi) \le C a^{D(         h_\xi(e) - h_\xi(g))} $$
for all $g \in G$. This is well-defined independently of the choice of $h_{\xi}$ by theorem \ref{thm:CP1} below. Here $g_*\eta(E)=\eta(g^{-1}E)$ for all Borel $E$. Our definition differs slightly from the one in [Co93] because we consider $g_*$ rather than $g^*$.   
\end{defn}

\begin{thm}[CP01, corollary 4.9] \label{thm:CP1}
 If $h_1, h_2 \in \sH_0$ then $\pi(h_1)=\pi(h_2)$ if and only if $||h_1-h_2||_\infty \le 64\delta$.
\end{thm}

\begin{thm}[Co93, corollaire 7.5]\label{thm:ps}
If $G$ is nonelementary and word hyperbolic and $\eta_1,\eta_2$ are $G$-quasiconformal measures on $\partial \Gamma$ of dimension $D_1, D_2$ respectively, then $D_1=D_2=\frac{e(\Gamma)}{\ln(a)}$ and $\eta_1$ is equivalent to $\eta_2$. In fact, both are equivalent to $D$-dimensional Hausdorff measure on $\partial \Gamma$ with respect to a natural metric.
\end{thm}

\begin{defn}
For $g \in G$ and $t \ge 0$ let $O(g,t)$ be the set of all $\xi \in \partial \Gamma$ such that there is a geodesic ray $r:[0,\infty) \to \Gamma$ such that: $r(0)=e$, $r(\infty)=\xi$ and for some $s \ge 0$, $d(r(s),g) \le t$. $O(g,t)$ is the {\bf shadow} of the ball $B(g,t)$ on $\partial \Gamma$.
\end{defn}

\begin{lem}[Co93, Proposition 6.1]\label{lem:shadow}
Let $\mu$ be a quasiconformal measure of dimension $D$ on $\partial \Gamma$. Then there exists constants $C \ge 1$ and $t_0 \ge 0$ such that for all $t>t_0$ and for all $g \in G$, 
$$C^{-1}a^{-|g|D} \le \mu(O(g,t)) \le Ca^{-|g|D + 2Dt}$$
where $|g|=d(g,e)$. 
\end{lem}


We can now prove proposition \ref{prop:generation}.
\begin{proof}[Proof of proposition \ref{prop:generation}]
The upper bound follows immediately from theorem \ref{thm:growth}. 

Let $\mu$ be a quasiconformal measure on $\partial \Gamma$. For $h\in\sH_0$ and $C>0$, let $O_h(e,C)$ be the set of all $\xi \in \partial \Gamma$ such that there exists a geodesic $r:(-\infty,\infty)\to\Gamma$ satisfying
$$r(+\infty)=\pi(h),\, d(r(0),e)\le C \textrm{ and } r(-\infty)=\xi.$$
Since $O_h(e,C) \subset O_h(e,C+1)$ and $\cup_{C >0} O_h(e,C) = \partial \Gamma - \{\pi(h)\}$, it follows that for some $C_0>0$, $\mu(O_h(e,C_0))>0$. Let $C_1=2C_0+8\delta$ and $S=B(e,C_1) \cap \{g\in G |\, h(g)=0\}$.

Recall that $S=B(e,C_1)\cap\{g \in G~|~h(g)=0\}$. We claim that 
$$O_h(e,C_0) \subset \bigcup_{g \in \Gen_n(h,S)}O(g,t)$$
for any $t >0$. So let $\xi \in O_h(e,C_0)$. Let $r:(-\infty,\infty)\to \Gamma$ be as above. So $r(+\infty)=\pi(h), r(-\infty)=\xi$ and $d(r(0),e)\le C_0$. Let $r':(-\infty,+\infty)\to \Gamma$ be a minimal $h$-gradient line with $r'(-\infty)=\xi$, $r'(+\infty)=\pi(h)$. Since $r$ and $r'$ have the same endpoints at infinity, they are within a Hausdorff distance of $4\delta$ of each other [CP01, Proposition 1.2]. So there exists an $s_1$ with $d(r'(s_1),r(0)) \le 4\delta$. Thus $d(r'(s_1),e)\le C_0+4\delta$. If $s_2$ is the number with $h(r'(s_2))=0$, it follows that $d(r'(s_2),e)\le C_1$. To see this, note that $C_0+4\delta \ge d(e,r'(s_1)) \ge |h(r'(s_1))|$ by the distance-like property of horofunctions. Since $h(r'(s_2))=0$ this implies that $d(r'(s_1),r'(s_2)) \le 4\delta + C_0$. The triangle inequality now implies $d(r'(s_2),e) \le C_1$ as claimed.

Thus $r'(s_2)\in S$. This shows that $\xi \in O(r'(s_2+n),t)$ for all $n\ge 0$ and all $t\ge 0$. Since $r'(s_2+n) \in \Gen_n(h,S)$ this proves the claim.

Let $t\ge t_0$ where $t_0$ is as in lemma \ref{lem:shadow}. That lemma implies 
\begin{eqnarray*}
|\Gen_n(h,S)| Ca^{-nD + 2Dt} &\ge& \sum_{g \in Gen_n(h,S)} \mu(O(g,t))\\
 &\ge& \mu\Big(\cup_{g \in \Gen_n(h,S)}O(g,t)\Big)\\
&\ge& \mu(O_h(e,C_0)).
\end{eqnarray*}
Thus
$$|\Gen_n(h,S)| \ge Ca^{nD - 2Dt}\mu(O_h(e,C_0)).$$
By theorems \ref{thm:growth} and \ref{thm:ps}, $a^{nD} = e^{e(\Gamma)n}$. This finishes the proof.
\end{proof}

\subsection{Periodicity}\label{sec:periodicity}

\begin{defn}
Let $\tM$ be the adjacency matrix of $\tsB$. So, the $(C,B)$-entry of $\tM$ equals $1$ if there is a directed edge in $\tsB$ from $B$ to $C$. It equals zero otherwise.
\end{defn}
\begin{lem}\label{lem:matrix}
For $n \ge 0$, $\veta_n=\tM\veta_{n+1}$. Here $\veta_n$ is as defined in \S \ref{s1} definition \ref{defn:veta}.
\end{lem}

\begin{proof}
Let $B, C \in \sB$. By lemma \ref{lem:etak},
\begin{eqnarray*}
\eta_n(C) = \int |\Par^{-1}(h)\cap C| d\eta_{n+1}.
\end{eqnarray*}
Because $\tsB$ has no multiple edges, the right hand side is the sum of $\eta_{n+1}(B)$ for all $B \in \sB$ such that there is a directed edge in $\tsB$ from $B$ to $C$. This uses lemma \ref{lem:determined}. In other words,
$$\eta_n(C) = \sum_{B \in \sB} M_{C,B} \eta_{n+1}(B).$$
\end{proof}
The next lemma follows directly from the definitions.
\begin{lem}\label{lem:4.16}
Let $h \in \sH$ and $S$ be a finite subset of $G$ with $h(s)=0$ for all $s \in S$. Let $v$ be the $\sB \times 1$ vector with $$v(B)=\Big| \big\{ s \in S ~|~ \Block(s^{-1}h)=B \big\}\Big|.$$ 
Then $|\Gen_n(h,S)| = ||M^nv||_1$ where $||\cdot||_1$ denotes the $l^1$-norm.
\end{lem}
 
\begin{defn}
A nonnegative $\sB \times 1$ vector $v$ {\bf has the same growth rate as the Cayley graph $\Gamma$} if there is a constant $C>0$ such that
\begin{eqnarray*}
C^{-1} e^{e(\Gamma)n} \le ||M^n v||_1 \le C e^{e(\Gamma)n}
\end{eqnarray*}
for all $n\ge 0$. 
\end{defn}

\begin{cor}
There exists a finite collection of nonnegative $\sB \times 1$ vectors $v_1, v_2,...,v_n$, each of which has the same growth rate as $\Gamma$, such that the following holds. For any $\eta \in M_R(\sH_0)$ there exists nonnegative coefficients $t_1,...,t_n$ such that
$$\veta_0 = t_1v_1 + \cdots + t_nv_n$$
where $\veta_0$ is the $\sB \times 1$ vector defined by $\veta_0(B)=\eta(B)$.
\end{cor}

\begin{proof}
For $h\in \sH_0$, let $m(h)$ be the number of elements $f \in G$ such that $h(f)=0$ and $d(f,g)\le C_1$ where $C_1$ is as in proposition \ref{prop:generation}. Let $v_h$ be the $\sB\times 1$ vector with $B$-entry given by
$$v_h(B) = \sum_{g} \frac{1}{m(g^{-1}h)}$$ 
where the sum is over all $g \in G$ such that $\Block(g^{-1}h)=B, d(g,e)\le C_1$ and $h(g)=0$. There are finitely many vectors of the form $v_h$. It follows from lemma \ref{lem:4.16} and proposition \ref{prop:generation} that $v_h$ has the same growth rate as $\Gamma$.

Define $F:R \to \R$ by $F(h,g^{-1}h)=\frac{1}{m(g^{-1}h)}$ if $d(g,e)\le C_1$, $h(g)=0$, and $\Block(g^{-1}h)=B$. $F(h,g^{-1}h)=0$ otherwise. The mass-transport principle (proposition \ref{prop:mtp}) applied to $F$ implies
$$\eta(B) = \int_{\sH_0} v_h(B) ~d\eta(h).$$
Since $B$ is arbitrary, this implies the corollary.

\end{proof}

\begin{thm}\label{thm:eigenvalue}
There exists a $p>0$ such that for all $\eta \in M_R(\sH_0)$, $\veta_0$ is an eigenvector of $M^p$ with eigenvalue $e^{e(\Gamma)p}$.
\end{thm}

\begin{proof}
Since $M$ is nonnegative, there exists a $p>0$ so that $M^p$ is (after ordering $\sB$ appropriately) lower block triangular and each diagonal block is either a primitive matrix or a zero matrix. Thus for any nonnegative vector $v$ and any positive integer $k$ either $M^{kp}v$ limits on the zero vector (as $k\to \infty$) or
$$E(v) = \lim_{k \to \infty} \frac{M^{kp}v}{||M^{kp}v||_1}$$
exists and is an eigenvector of $M^p$. 

Let $v_1,\dots,v_n$ be as in the previous corollary. Since each $v_i$ has the same growth rate as $\Gamma$, $E(v_i)$ exists and has eigenvalue $e^{e(\Gamma)p}$. It will be simpler to work with normalized vectors. So, for $n \ge 0$, let $\eta'_n = \frac{\veta_n}{||\veta_n||_1}$. After scaling if necessary, we may assume that $||v_i||_1=1$ for all $i$.

For each $k\ge 0$, there exists nonnegative coefficients $t_{k,1}, t_{k,2},..., t_{k,n}$ so that
$$\eta'_{kp} = t_{k,1}v_1+ t_{k,2} v_2 + \cdots + t_{k,n}v_n.$$
 Since $t_{k,1} + \cdots + t_{k,n} = 1$, $0 \le t_{k,i} \le 1$ for all $k,i$. By lemma \ref{lem:matrix},
$$\veta_0 = \frac{M^{kp}\eta'_{kp}}{||M^{kp}\eta'_{kp}||_1} 
= \sum_{i=1}^n t_{k,i}\frac{M^{kp}v_i}{||M^{kp}v_i||_1} \frac{||M^{kp}v_i||_1}{||M^{kp}\eta'_{kp}||_1}.$$
Since $\eta'_{kp}$ is a convex sum of the vectors $\{v_i\}_{i=1}^n$ it has the same growth rate as $\Gamma$. In fact, there is a constant $C>0$ such that 
$$C^{-1} e^{e(\Gamma)kp} \le ||M^{kp}\eta'_{kp}||_1 \le C e^{e(\Gamma)kp}$$
for all $k \ge 0$ (e.g., take $C$ to be the maximum over all such constants occuring in the related inequalities for $v_1,\dots,v_n$). Thus there exists a subsequence $\{k_j\}$ of $\N$ such that for each $i$, the ratio $\frac{||M^{k_jp}v_i||_1}{||M^{k_jp}\eta'_{k_jp}||_1}$ converges as $j \to \infty$ to some constant $C_i$. Observe that 
\begin{eqnarray*}
\Big|\Big|\veta_0 - \sum_{i=1}^n t_{k_j,i} C_i E(v_i)\Big|\Big|_1 &\le & \sum_{i=1}^n t_{k_j,i} \Big|\Big| \frac{M^{k_jp}v_i}{||M^{k_jp}v_i||_1} \frac{||M^{k_jp}v_i||_1}{||M^{k_jp}\eta'_{k_jp}||_1} - C_i E(v_i) \Big|\Big|_1.
\end{eqnarray*}
The right hand side tends to zero as $j \to \infty$. By passing to another subsequence of $\{k_j\}$ if necessary, we may assume that for each $i$, $t_{k_j,i}$ converges (as $j\to \infty$) to a constant ${\bar t}_i$. Thus we have shown that
\begin{eqnarray*}
\veta_0 = \sum_{i=1}^n {\bar t_{i}} C_i E(v_i).
\end{eqnarray*}
Since each $E(v_i)$ is an eigenvector of $M^p$ with eigenvalue $e^{e(\Gamma)p}$ this proves the theorem.
\end{proof}

\begin{defn}\label{defn:alpha}
Let $\beta: \sH \to \sH$ be the map $\beta(h)=h-1$. So $\beta(\sH_k) =\sH_{k+1}$ for all $k \in \Z$. 

If $I,J \subset \Z$ are finite sets and $I \subset J$ then, as in lemma \ref{restriction}, the restriction map $\Res:M_R(\sH_J) \to M_R(\sH_I)$ is an isomorphism. Therefore, if $I$ and $J$ are {\it any} finite subsets of $\Z$, there is a natural isomorphism from $\kappa: M_R(\sH_I) \to M_R(\sH_J)$ obtained by following the inverse of the normalized restriction map from $M_R(\sH_I) \to M_R(\sH_{I \cup J})$ with the normalized restriction map from $M_R(\sH_{I \cup J}) \to M_R(\sH_J)$. 

Let $\eta \in M_R(\sH_I)$ and let $\kappa(\eta) \in M_R(\sH_{I+1})$ be the corresponding measure. Here $I+1=\{i+1|\, i\in I\}$. Define $\alpha: M_R(\sH_I) \to M_R(\sH_I)$ by $\alpha\eta =\kappa \eta \circ\beta$.
\end{defn}

The next corollary follows immediately from the above theorem and proposition \ref{prop:determinism}.
\begin{cor}\label{cor:periodic}
There exists a $p>0$ such that for every finite set $I \subset \Z$ and every $\eta \in M_R(\sH_I)$, $\alpha^p\eta=\eta$.
\end{cor}

\subsection{$M_R(\sH_*)$ is isomorphic to $M_R(\sH_0)$}\label{sub:alpha}
Let $\sH_*=\{h \in \sH|\, h(e) \le 0\}=\sH_{[0,\infty)}$. It is endowed with the equivalence relation induced by the action of $G$ on $\sH$ restricted to $\sH_*$. Let $M_R(\sH_*)$ be the set of all Borel probability measures on $\sH_*$ that are invariant under the partial transformations of this relation. 

\begin{lem}\label{lem:isomorphic}
The restriction map $\Res:M_R(\sH_*) \to M_R(\sH_0)$, obtained by restricting $\eta \in M_R(\sH_*)$ to $\sH_0$ and normalizing, is an isomorphism.
\end{lem}

\begin{proof}
Given $\eta \in M_R(\sH_0)$, define $\eta_k$ as in subsection \ref{s1}. Define $\eta_* \in M_R(\sH_*)$ by
$$\eta_* = \frac{1}{C} \sum_{k \ge 0} \eta_k$$
where $C = \sum_{k\ge 0} \eta_k(\sH_k)$. By theorem \ref{thm:eigenvalue} and lemma \ref{lem:matrix}, it follows that there exists a constant $C_0>0$ such that 
$$C^{-1}_0e^{-e(\Gamma)k} \le \eta_k(\sH_k) =||\veta_k||_1 \le C_0e^{-e(\Gamma)k}$$
for all $k\ge 0$. Therefore $0 < C < \infty$ and $\eta_*$ is a well-defined probability measure. It is easy to check from the definition of $\eta_k$ that $\eta_*$ really is relation-invariant. So the map $\eta \to \eta_*$ is the inverse of the restriction map.
\end{proof}

\begin{defn}
Define $\alpha_*:M_R(\sH_*)\to M_R(\sH_*)$ by $\alpha_* = \Res^{-1}\circ\alpha\circ\Res$. 

Equivalently, for $\eta \in M_R(\sH_*)$ and $E$ a Borel subset of $\sH_*$ define
$$\alpha_* \eta(E) = \frac{\eta(\beta(E))}{\eta(\sH_{[1,\infty)})} = \frac{\beta^{-1}_*\eta(E)}{\eta(\sH_{[1,\infty)})}$$
where $\beta:\sH\to\sH$ is the map $\beta(h)=h-1$. 

Equivalently, for $\eta \in M_R(\sH_*)$ let $\eta^1$ be the normalized restriction of $\eta$ to $\sH_{[1,\infty)}$. Then $\alpha_* \eta = \eta^1 \circ \beta$. Similarly, if $\eta^k$ is the normalized restriction of $\eta$ to $\sH_{[k,\infty)}$ then $\alpha_*^k\eta = \eta^k \circ \beta^k$. So if $E \subset \sH_*$ is Borel then 
$$\alpha_*^k \eta(E) = \frac{\eta(\beta^k(E))}{\eta(\sH_{[k,\infty)})}= \frac{\beta^{-k}_*\eta(E)}{\eta(\sH_{[k,\infty)})}.$$

By abuse of notation, we will write $\alpha=\alpha_*$. 
\end{defn}

\subsection{Quasiconformal measures}\label{sub:qc}
In this section we show that if $\eta \in M_R(\sH_*)$ is $\alpha$-invariant then its projection to the boundary is quasiconformal. We use this to finish the proof of theorem \ref{thm1}.

\begin{defn}\label{def:phi}

For $g \in G$ and $h \in \sH$, define $\phi(g)h \in \sH$ by 
$$\big(\phi(g)h\big)(f)=h(g^{-1}f) - h(g^{-1}) + h(e).$$
Observe that for all $k \in \Z$, $\phi(g):\sH_k \to \sH_k$. In particular, this defines an action of $G$ on $\sH_*$.


\end{defn}

\begin{lem}
If $\eta \in M_R(\sH_*)$, $\alpha\eta=\eta$ and $k\ge 0$ then 
$$\eta(\sH_{[k,\infty)}) = e^{-e(\Gamma)k}.$$
\end{lem}
\begin{proof}
Let $\veta_k$ be the $\sB\times 1$ vector with $B$-entry equal to $\eta(B \cap \sH_k)$. Since $\eta$ is $\alpha$-invariant, lemma \ref{lem:matrix} implies that $\veta_k$ is an eigenvector of $M$. Theorem \ref{thm:eigenvalue} implies that the eigenvalue is $e^{e(\Gamma)}$. Lemma \ref{lem:matrix} also shows that $M\veta_{k+1}=\veta_k$ for all $k\ge 0$. Since $||\veta_k||_1=\eta(\sH_k)$ this implies that
$$\eta(\sH_k)=||\veta_k||_1=||M\veta_{k+1}||_1 =e^{e(\Gamma)}||\veta_{k+1}||=e^{e(\Gamma)}\eta(\sH_{k+1}).$$
Thus,
$$\eta(\sH_k) e^{e(\Gamma)k} = \eta(\sH_0).$$
Therefore
\begin{eqnarray*}
1 = \sum_{k\ge 0} \eta(\sH_k) = \frac{\eta(\sH_0)}{1- e^{-e(\Gamma)}}.
\end{eqnarray*}
Thus $\eta(\sH_0)=1- e^{-e(\Gamma)}$ and $\eta(\sH_k)=\eta(\sH_0)e^{-e(\Gamma)k} = e^{-e(\Gamma)k} - e^{-e(\Gamma)(k+1)}$. So,
$$\eta(\sH_{[k,\infty)}) = \sum_{i=k}^\infty \eta(\sH_i) = e^{-e(\Gamma)k}.$$
\end{proof}

\begin{lem}\label{lem:precise}
Let $\eta \in M_R(\sH_*)$ be $\alpha$-invariant. Then for any $g \in G$,
$$\frac{d\phi(g)_*\eta}{d\eta}(h) = e^{e(\Gamma)(h(e) -h(g)) }.  $$
\end{lem}
\begin{proof}
If $h \in \sH_i \cap g\sH_j$ then $h(e)=-i$ and $h(g)=-j$. Thus it suffices to prove that if $f: \sH \to \R$ is any Borel function and $i,j \in \Z$ then
$$\int_{\sH_i \cap g\sH_j} f(h) \,d\phi(g)_*\eta(h) = e^{-e(\Gamma)(i-j)} \int_{\sH_i \cap g\sH_j} f(h) \,d\eta(h).$$
This is trivial if $i< 0$. So assume $i \ge 0$. Rewrite the left hand side of the above equation as follows.
\begin{eqnarray*}
\int_{\sH_i \cap g\sH_j} f(h) \,d\phi(g)_*\eta(h) &=& \int_{\phi(g)^{-1}[\sH_i \cap g\sH_j]} f(\phi(g)h) \, d\eta(h)\\
&=& \int_{\beta^{i-j}g^{-1}[\sH_i \cap g\sH_j]} f(\phi(g)h) \,d\eta(h)\\
&=& \int_{\beta^{i-j}g^{-1}[\sH_i \cap g\sH_j]} f(\beta^{j-i} g h) \,d\eta(h)\\
&=& \int_{g^{-1}\sH_i \cap \sH_j} f(gh) \,d\beta^{j-i}_*\eta(h).
\end{eqnarray*}
The third equation occurs since if $h \in \beta^{i-j}g^{-1}[\sH_i \cap g\sH_j] = g^{-1}\sH_{2i-j} \cap \sH_i$ then $\phi(g)h = \beta^{j-i}gh$.

{\bf Case 1}. Suppose $2i\ge j$. Then for any $h\in g^{-1}\sH_i \cap \sH_j$, both $h$ and $gh$ are in $\sH_{[j-i,\infty)}$. Note $\beta^{j-i}_* \eta \in M_R(\sH_{[j-i,\infty)})$. So, relation-invariance implies that 
\begin{eqnarray*}
\int_{g^{-1}\sH_i \cap \sH_j} f(gh) \,d\beta^{j-i}_*\eta(h) &=& \int_{\sH_i \cap g\sH_j} f(h) \,d\beta^{j-i}_*\eta(h).
\end{eqnarray*}
By definition, $\beta^{j-i}_*\eta(E) = \alpha^{i-j}\eta(E) \eta(\sH_{[i-j,\infty)})$ for any Borel $E \subset \sH_*$ (see \S \ref{sub:alpha} for the definition). Thus the above equals 
$$ \eta(\sH_{[i-j,\infty)})\int_{\sH_i \cap g\sH_j} f(h) \,d\alpha^{i-j}\eta.$$
Since $\eta$ is $\alpha$-invariant, the previous lemma now implies
$$\int_{\sH_i \cap g\sH_j} f(h) \,d\phi(g)_*\eta(h) = e^{-e(\Gamma)(i-j)} \int_{\sH_i \cap g\sH_j} f(h) \,d\eta(h).$$
This finishes the lemma in case $2i\ge j$.

{\bf Case 2}. Suppose $2i \le j$. Let $k = j-i \ge 0$. Let $\eta^k \in M_R(\sH_{[-k,\infty)})$ be the measure whose restriction to $M_R(\sH_{[0,\infty)})$ is $\eta$. This uniquely defines $\eta^k$ by an argument similar to the proof of lemma \ref{restriction}. 

By definition, if $E \subset \sH_{[0,\infty)}$ then
$$\alpha^{-k}\eta(E) = \eta^k(\beta^{-k}E).$$
If $E \subset \sH_{[k,\infty)}$ then $\beta^{-k}(E) \subset \sH_*$. Therefore, $\eta^k(\beta^{-k}(E)) =\eta(\beta^{-k}(E))\eta^k(\sH_*)=\beta^k_*\eta(E)\eta^k(\sH_*)$. By an argument similar to the one in the previous lemma, $\eta^k(\sH_*) =  e^{-e(\Gamma)k}$. Thus for $h \in \sH_{[k,\infty)}$,
\begin{eqnarray*}
d(\beta_*^{j-i}\eta )(h) &=& d(\beta^k_*\eta)(h)\\
& =& \frac{d (\beta^k_* \eta^k ) (h)}{\eta^k(\sH_*)}\\
&=& \frac{d\alpha^{-k}\eta(h) }{\eta^k(\sH_*)} = e^{-e(\Gamma)(i-j)}d\eta(h).
\end{eqnarray*}
Now $2i \le j$ implies $g^{-1}\sH_i \cap \sH_j \subset \sH_{[j-i,\infty)} = \sH_{[k,\infty)}$. So 
\begin{eqnarray*}
\int_{\sH_i \cap g\sH_j} f \, d\phi(g)_*\eta &=& \int_{g^{-1}\sH_i \cap \sH_j} f(gh) \,d(\beta^{j-i}_*\eta)(h)\\
&=&e^{-e(\Gamma)(i-j)}\int_{g^{-1}\sH_i \cap \sH_j} f(gh)\, d\eta(h).
\end{eqnarray*}
Since $j\ge 2i \ge i \ge 0$, for all $h \in g^{-1}\sH_i \cap \sH_j$ we have that $h$ and $gh$ are in $\sH_*$. So $g$ restricted to $g^{-1}\sH_i \cap \sH_j$ is a partial transformation. Since $\eta$ is relation-invariant, this implies that the above equals
$$e^{-e(\Gamma)(i-j)}\int_{\sH_i \cap g\sH_j} f(h)\, d\eta(h).$$
This finishes the case $2i\le j$ and hence, the lemma.
\end{proof}

\begin{lem}
The above lemma remains true if $\eta \in M_R(\sH_0)$ and is $\alpha$-invariant.
\end{lem}
\begin{proof}
Let $\eta_* \in M_R(\sH_*)$ be such that the normalized restriction of $\eta_*$ to $\sH_0$ equals $\eta$. The previous lemma applied to $\eta_*$ implies this lemma. 
\end{proof}

\begin{lem}\label{lem:quasiconformal}
If $\eta \in M_R(\sH_0)$ is $\alpha$-invariant then $\pi_*\eta$ is quasiconformal. 
\end{lem} 

\begin{proof}
This follows immediately from the previous lemma and definition \ref{def:qc}.
\end{proof}

To prove theorem \ref{thm1} we will need:

\begin{thm}[CP01, proposition 5.5]\label{thm:CP2}
 The map $\pi: \sH_0 \to \partial \Gamma$ is uniformly finite-to-1. That is, there exists a constant $C = C(G,A)$ such that for every $\xi \in \partial \Gamma$, $|\pi^{-1}(\xi)| \le C$.
\end{thm}


\begin{proof}[Proof of theorem \ref{thm1}]

Let $\eta_1,...,\eta_n$ be any collection of distinct ergodic measures in $M_R(\sH_0)$. It suffices to show that there is a constant $C >0$ (depending only on $(G,A)$) such that $n \le C$. For each $i$, let 
$$\nu_i = \frac{1}{p} \sum_{j=0}^{p-1} \alpha^j \eta_i.$$
Observe that at least $n/p$ of the measures $\nu_1,...,\nu_n$ are distinct. So, after renumbering if necessary, we may assume that there is a number $m \ge n/p$ such that $\nu_1,...,\nu_m$ are distinct. By corollary \ref{cor:periodic}, each $\nu_i$ is $\alpha$-invariant. The construction implies that they are pairwise mutually singular. So there exists Borel sets $E_1,\dots,E_m \subset \sH_0$ such that $\nu_i(E_j) = \delta^i_j$ for all $1\le i,j \le m$, where $\delta^i_j$ is the Dirac-$\delta$ symbol.

By lemma \ref{lem:quasiconformal} each $\pi_*(\nu_i)$ is quasiconformal. By theorem \ref{thm:ps}, the pushforward measures $\pi_*\nu_i$ on $\partial \Gamma$ are all equivalent. Therefore, $\pi_*\nu_i(\pi(E_j))=1$ for all $i,j$. This implies that there is a point $\xi \in \partial \Gamma$ such that the inverse image $\pi^{-1}(\xi)$ has nontrivial intersection with all of the sets $E_1,...,E_m$. By theorem \ref{thm:CP2}, the number of preimages of $\xi$ is bounded by some constant $C>0$ that depends only on $(G,A)$. Thus $n/p\le m \le C$ which implies $n \le pC$.
\end{proof}

\section{Proof of Theorem \ref{thm:main}}\label{sec:part2}
Let $(X,\mu)$ be a standard Borel probability space on which $G$ acts by measure-preserving Borel transformations. $G$ acts diagonally on the product $\sH \times X$. This action induces an equivalence relation on $\sH\times X$ and, by restriction, on $\sH_I \times X$ for any $I \subset \Z$ ($\sH_I$ is defined in subsection \ref{s1}). Let $M_R(\sH_I\times X)$ denote the space of Borel probability measures on $\sH_I\times X$ that are invariant under this relation.

As in subsection \ref{s1}, the map $\Res:M_R(\sH_{[-k,k]} \times X) \to M_R(\sH_0\times X)$ obtained by restricting $\omega \in M_R(\sH_{[-k,k]} \times X)$ to $\sH_0\times X$ and normalizing is an isomorphism. We can now define a map $\alpha:M_R(\sH_0\times X) \to M_R(\sH_0\times X)$ in a manner analogous to definition \ref{defn:alpha}. That is, let $\kappa \omega \in M_R(\sH_1 \times X)$ be the measure obtained from $\omega \in M_R(\sH_0\times X)$ by following the inverse of the restriction map from $M_R(\sH_0\times X)\to M_R(\sH_{[0,1]} \times X)$ with the restriction map from  $M_R(\sH_{[0,1]} \times X) \to  M_R(\sH_1\times X)$. Then for any Borel $E\subset \sH_0\times X$,
$$\alpha\omega(E) := \kappa \omega(\beta \times 1_X(E))$$
where $1_X: X \to X$ denotes the identity map and $\beta:\sH\to\sH$ is the map $\beta(h)=h-1$. In a similar manner, we can define $\alpha: M_R(\sH_* \times X) \to M_R(\sH_* \times X)$. 

In the next section we will prove:
\begin{thm}\label{thm:step1}
There exists a $q>0$ such that the following holds. Let $\eta \in M_R(\sH_0)$. If $\lambda \in M_R(\sH_0\times X)$ is absolutely continuous to $\eta \times \mu$ then $\alpha^i\lambda=\lambda$ for some $0< i \le q$.
\end{thm}

\begin{lem}
Let $\nu$ be a $G$-quasiconformal measure on $\partial \Gamma =\partial G$. Let $G$ act on a standard Borel probability space $(X,\mu)$ by measure-preserving transformations. If the action of $G$ on $(X,\mu)$ is ergodic then the diagonal action of $G$ on $(\partial G \times X, \nu \times \mu)$ is also ergodic.
\end{lem}

\begin{proof}
It suffices to show that if $F:\partial G \times X \to [0,1]$ is any $G$-invariant measurable function then $F$ is constant on a conull set. For $b \in \partial G$, let $F_b: X \to [0,1]$ be the function $F_b(x)=F(b,x)$. Because $F$ is $G$-invariant, the map $b \to F_b$ from $\partial G$ into $L^2(X)$ is $G$-equivariant where the action of $G$ on $L^2(X)$ is $(g,f) \to f \circ g^{-1}$. 

The space $L^2(X)$ is a separable coefficient $G$-module. It follows from [Ka03, theorem 3] that the map $b \to F_b$ must be constant on a conull subset of $\partial G$. Actually, that result applies to the Poisson boundary of $G$ rather than the Gromov boundary. However, it is well-known that the two boundaries coincide (see e.g., [Ka00]). Therefore, up to measure zero, $F$ depends only on its second argument, i.e., $F(b,x)=f(x)$ for some function $f:X \to [0,1]$. Because $F$ is $G$-invariant, $f$ must be $G$-invariant as well. Since the action $G \curvearrowright (X,\mu)$ is ergodic, it follows that $f$ must be constant on a conull set. Hence $F$ is constant on a conull set.
\end{proof}

\begin{remark}
An earlier version of [Ka03, theorem 3] was proven in [BM02, theorem 6] (see also [Mo01, theorem 11.1.3]).
\end{remark}



\begin{proof}[Proof of theorem \ref{thm:main}]
We may assume that $\eta$ is $\alpha$-invariant. To see this, let $\nu=\frac{1}{p} \sum_{i=0}^{p-1} \alpha^i\eta$. By corollary \ref{cor:periodic}, $\nu$ is $\alpha$-invariant. Since $\eta \times \mu$ is absolutely continuous to $\nu \times \mu$, if the theorem is true for $\nu$ then it must be true for $\eta$. Thus after replacing $\eta$ by $\nu$ if necessary, we may assume that $\eta$ is $\alpha$-invariant.

 Suppose that $$\eta \times \mu = \sum_{i=1}^k t_i \omega_i$$
where $t_i \ge 0$, $\omega_i \in M_R(\sH_0\times X)$ and the measures $\omega_1,...,\omega_k$ are pairwise mutually singular. We do not assume that $\omega_1,...,\omega_k$ are ergodic. It suffices to prove that $k$ is bounded by a universal constant. By employing theorem \ref{thm:step1}, we may assume that each $\omega_i$ is $\alpha$-invariant.

For $h\in \sH$, let $\bh \in \sH_0$ be defined by $\bh(g)=h(g)-h(e)$. For $g \in G, h \in\sH_0,$ and $x\in X$, let $\tphi(g)(h,x)=(\overline{gh}, gx)$. Let $\tpi : \sH_0 \times X \to \partial \Gamma \times X$ be the map $\tpi(h,x)=(\pi( h),x)$. $\tpi$ is $G$-equivariant in the sense that $\tpi(\tphi(g)(h,x))=g \tpi(h,x)$ where $G$ acts on $\partial \Gamma \times X$ diagonally.

By lemma \ref{lem:quasiconformal}, $\pi_*(\eta)$ is quasiconformal. By the previous lemma, $\pi_*(\eta)\times \mu$ is ergodic. Since each $\omega_i$ is $\alpha$-invariant, each $\omega_i$ is quasiinvariant under the $\tphi$-action of $G$. So $\tpi_*(\omega_i)$ is $G$-quasiinvariant. Because $\tpi_*(\omega_i)$ is absolutely continuous to $\tpi_*(\eta \times \mu) = \pi_*(\eta)\times \mu$ which is ergodic, this implies that $\tpi_*(\omega_i)$ is equivalent to $\pi_*(\eta)\times \mu$.

Since the measures $\omega_1,...,\omega_k$ are pairwise mutually singular, there exists pairwise disjoint Borel sets $E_1,...,E_k$ such that $\omega_i(E_i) = 1$ for all $i$. Because $\tpi_*(\omega_i)$ is equivalent to $\pi_*(\eta)\times\mu$, it follows that 
$$\pi_*\eta\times\mu\Big(\bigcap_{i=1}^k \tpi(E_i)\Big)=1.$$
So there exists a point $(b,x) \in \bigcap_{i=1}^k \tpi(E_i)$. Since $\tpi:\sH_0 \times X \to \partial \Gamma \times X$ is uniformly finite-to-1 (by theorem \ref{thm:CP2}) and since the sets $E_i$ are pairwise disjoint, this implies that $k$ is bounded by a constant depending only on $(G,A)$.
\end{proof}

\section{Components of $\eta \times \mu$ are virtually $\alpha$-invariant}\label{sec:part1}

A block $B \in \sB$ is called {\bf recurrent} if there is a directed cycle containing it. Here we are considering $\sB$ as a directed graph (definition \ref{defn:tsB}). Let $\sB_r\subset \sB$ denote the set of recurrent blocks. In section \ref{section:horospherical} we prove the following. 
\begin{lem}\label{lem:key}[Key Lemma]
For any $\eta \in M_R(\sH_0)$ such that $\alpha \eta =\eta$ there exist recurrent blocks $B, C \in \sB_r$, $\bg \in G$ and nonnegative integers $s\ne t$ such that
$$\eta\big(\bg\sH_{s} \cap B \cap \bg C\big)\eta\big(\bg\sH_{t} \cap B \cap \bg C\big)>0.$$
\end{lem}
In this section, we prove theorem \ref{thm:step1} assuming the above lemma. From here on, suppose that $\eta \in M_R(\sH_0)$ is fixed. Of course, if theorem \ref{thm:step1} is true for the measure $\nu=\frac{1}{p} \sum_{i=0}^{p-1} \alpha^p \eta$ in place of $\eta$, then it must be true for $\eta$ too. By corollary \ref{cor:periodic}, $\nu$ is $\alpha$-invariant. Therefore, we may assume, after replacing $\eta$ by $\nu$ if necessary, that $\eta$ is $\alpha$-invariant.

\begin{lem}\label{lem:singular}
Suppose that there is a measure $\lambda_0 \in M_R(\sH_0\times X)$ that is absolutely continuous to $\eta \times \mu$ and a number $q >0$ such that $\alpha^i \lambda_0 \ne \lambda_0$ for any $0< i \le q$. Then there exists a measure $\lambda \in M_R(\sH_0\times X)$ that is absolutely continuous to $\eta \times \mu$ such that $\lambda, \alpha \lambda, ..., \alpha^{q}\lambda$ are pairwise mutually singular.
\end{lem}

\begin{proof}
By the Krein-Milman theorem, there exists a probability measure $\nu$ on $M_R^e(\sH_0\times X)$, the space of ergodic measures in $M_R(\sH_0\times X)$, such that
$$\eta\times \mu = \int \omega \, d\nu(\omega).$$

Let ${\bar Y} \subset M^e_R(\sH_0\times X)$ be the set of $\omega \in M^e_R(\sH_0\times X)$ such that $\alpha^i \omega = \omega$ for some $i$ with $0< i \le q$. The hypothesis implies $\nu({\bar Y}) < 1$. Let $Y_0 = M_R(\sH_0 \times X) - {\bar Y}$.

Suppose that for some $i$ with $q > i \ge 0$, a set $Y_i$ has been defined so that $Y_i \subset M^e_R(\sH_0\times X) - {\bar Y}$, $\nu(Y_i)>0$ and $Y_i, \alpha(Y_1),...,\alpha^i(Y_i)$ are pairwise disjoint. We claim that $Y_i$ contains a nonnegligible subset $Y_{i+1} \subset Y_i$ such that $Y_{i+1}, \alpha(Y_{i+1}),...,\alpha^{i+1}(Y_{i+1})$ are pairwise disjoint. If this were not true, then it follows that for every measurable $Z \subset Y_i$ with $\nu(Z)>0$, $Z \cap \alpha^{i+1}Z \ne \emptyset$. Applying this to $Z-\alpha^{i+1}(Z)$ we see that $\nu(Z \Delta \alpha^{i+1}Z)=0$ for every $Z \subset Y_i$ with $\nu(Z)>0$. But this implies that $\alpha^{i+1}$ is the identity map on $Y_i$, contradicting that $i < q$ and $Y_i \cap {\bar Y} =\emptyset$. 

Thus there exists a set $Y_q \subset M_R^e(\sH_0\times X)$ so that the sets $Y_q,\alpha(Y_q),...,\alpha^q(Y_q)$ are pairwise disjoint and $\nu(Y_q)>0$. Let
$$\lambda = \frac{1}{\nu(Y_q)} \int_{Y_q} \omega \, d\nu(\omega).$$
Then $\lambda$ satisfies the conclusions.  
\end{proof}

Let $Q$ be the maximum value of $s$ or $t$ that occurs in the key lemma. Suppose, for a contradiction, that there exists a measure $\lambda \in M_R(\sH_0\times X)$ such that $\lambda$ is absolutely continuous to $\eta \times \mu$ and $\alpha^i\lambda \ne \lambda$ for any $1\le i \le Q$. By lemma \ref{lem:singular}, we can assume that $\lambda, \alpha\lambda,...,\alpha^Q\lambda$ are mutually singular. This implies that there exists pairwise disjoint sets $E_i \subset \sH_0\times X$ such that $(\alpha^i\lambda) (E_j) =\delta^i_j$ where $\delta^i_j$ is the Dirac $\delta$-symbol. It follows that for $(\eta \times \mu)$-a.e. $(h,x) \in E_0$, if $g \in G$ and $gh \in \sH_i$ for some $i$ with $0 \le i \le Q$ then $(\beta^{-i} \times 1_X)(gh,gx) \in E_i$ where $\beta(h)=h-1$ is as defined in \S \ref{sec:periodicity}.

It will be necessary to approximate each $E_i$ by a ``finitely determined'' set. This is explained next.

\begin{defn}
For $h \in \sH_0$ and $r\ge 0$, let
$$[h]_r = \big\{h' \in \sH_0~\big|~ \Block(\Par^n(h'))=\Block(\Par^n(h)) \, \forall 0\le n \le r \big\}$$
be the {\bf cylinder set around $h$ of order $r$}.
\end{defn}

\begin{defn}
For $r>0$, a Borel set $E \subset \sH_0$ is {\bf $r$-determined} if for all $h\in E$ and for all $h' \in \sH_0$ such that 
$$\Block(\Par^n(h))=\Block(\Par^n(h'))$$
for all $ 0\le n \le r$, $h'\in E$. Equivalently, $E$ is $r$-determined if it is a union of cylinder sets of order $r$. We will say that a Borel set $E \subset \sH_0 \times X$ is {\bf $r$-determined} if it is a union of sets of the form $E' \times Y$ where $E' \subset \sH_0$ is $r$-determined and $Y \subset X$.
\end{defn}

\begin{lem}\label{lem:approx}
Let $\lambda \in M_R(\sH_0\times X)$. If $E$ is any Borel set in $\sH_0\times X$ and $\epsilon>0$ then there exists a set $E' \subset \sH_0\times X$ that is $r$-determined (for some $r$) such that $\lambda(E'\Delta E)<\epsilon$. 
\end{lem}

\begin{proof}
Give $X$ a compact topology compatible with its Borel structure. Since $\lambda$ is a Borel measure, it is regular. So for every $\epsilon>0$ there exists an open set $O \subset \sH_0 \times X$ such that $E \subset O$ and $\lambda(O-E)<\epsilon$.

By theorem \ref{thm:CPmain}, the map which associates to $h \in \sH_0$ the sequence $n \to \Block(\Par^n(h))$ is a homeomorphism onto a subshift of finite type over $\N$. Therefore every open set of $\sH_0$ is a union of cylinder sets. Thus open subsets of $\sH_0\times X$ are unions of sets of the form $[h]_r \times U$ where $U$ is open in $X$. This implies the lemma.
\end{proof}

Later we will use the above lemma to approximate each set $E_i$ with an $r$-determined set $E'_i$. Now let $B,C \in \sB_r$, $\bg \in G$ and $s,t \in \Z$ be as in the key lemma. For each $r>0$ we will define a partial transformation $\psi_r:$ Dom$(\psi_r) \to $Im$(\psi_r)$ where Dom$(\psi_r) \subset \Par^{-r}(B \cap \bg C)$ and Im$(\psi_r) \subset \Par^{-r}(\bg^{-1} B \cap  C)$. Roughly speaking, the idea of the proof is to obtain a contradiction by studying the sets $E'_i \cap \bphi^{-1}_r(E'_j)$ and their $(\eta \times \mu)$-values where $\bphi_r$ is a natural extension of $\psi_r$ to Dom$(\psi_r) \times X$.

\begin{defn}\label{defn:psi}
First define
$$K_r:=\big\{g \in G~\big|~ gh=Par^r(h) \textrm{ for some } h \in \Par^{-r}(B \cap \bg C) \big\}.$$
$$L_r:=\big\{g \in G~\big|~ gh=Par^r(h) \textrm{ for some } h \in  \Par^{-r}(\bg^{-1} B \cap  C) \big\}.$$
The key lemma implies $B \cap \bg C$ is nonempty. Thus $\bg^{-1}(B \cap \bg C)=\bg^{-1} B \cap  C$ is nonempty, too. Since $B$ and $C$ are recurrent, $\Par^{-r}(B\cap \bg C)$ and $\Par^{-r}(\bg^{-1}B \cap C)$ are nonempty. So $K_r$ and $L_r$ are nonempty.

Let $f_r:$ dom$(f_r) \to $ rng$(f_r)$ be a bijection with dom$(f_r) \subset K_r$ and rng$(f_r) \subset L_r$. Let 
$$\sN^r=\big\{ h \in \Par^{-r}(B \cap \bg C) \cap \sH_0~|~ gh=\Par^r(h) \textrm{ for some } g \in dom(f_r)\big\}.$$
Define $\psi_r: \sN^r \to \sH$ by $\psi_r(h)= f_r(  g_h )^{-1}\bg^{-1}\Par^r(h)$ where $g_h$ is such that $g_hh = \Par^r(h)$. The only properties of $\psi_r$ that we will use are contained in the next lemma.
\end{defn}
\begin{lem}
\begin{enumerate}
\item $\Par^r(\psi_r(h)) = \bg^{-1}\Par^r(h) \in \bg^{-1} B \cap C$.
\item $\psi_r$ is injective. Thus it is a partial transformation of $\sH$.
\item If $h \in \sN^r \cap \sH_0$, then $\psi_r(h) \in \sH_i$ for some $i$ with $|i| \le d(\bg ,e)$. 
\end{enumerate}
\end{lem}
\begin{proof}
{\bf 1.} First note that $\forall h \in g_0^{-1}B \cap C, ~ \forall g \in L_r$, $\Par^r(g^{-1}h)=h$. To see this, note that by definition of $L_r$, there is some $h_0 \in \Par^{-r}(g_0^{-1}B \cap C)$ such that $gh_0=\Par^r(h_0)$. Thus, if $h=\Par^r(h_0)$ then $g^{-1}h=h_0$ so $\Par^r(g^{-1}h)=\Par^r(h_0)=h$. The general case now follows from lemma \ref{lem:determined}. 

Now if $h\in \sN^r$ then $g_0^{-1}\Par^r(h)\in g_0^{-1}B \cap C$ and $f_r(g_h^{-1}) \in L_r$. So $\Par^r(f_r(g_h)^{-1}g_0^{-1}\Par^r(h))=g_0^{-1}\Par^r(h)$.

{\bf 2.} Suppose for some $h_1,h_2 \in \sN^r$ that $\psi_r(h_1)=\psi_r(h_2)$. Then $\Par^r(\psi_r(h_1))=\Par^r(\psi_r(h_2))$. By part (1) above this implies $\Par^r(h_1)=\Par^r(h_2)$. But for $i=1,2$, $\psi_r(h_i)=f_r(g_{h_i})^{-1}g_0^{-1}\Par^r(h_i)$. So this implies that $f_r(g_{h_1})=f_r(g_{h_2})$. But $f_r$ is a bijection. So $g_{h_1}=g_{h_2}$. By definition, $h_1=g_{h_1}^{-1}\Par^r(h_1)=g_{h_2}^{-1}\Par^r(h_2)=h_2$.

{\bf 3.} To see this, observe that $\Par^r(h) \in \sH_r$. The distance-like property of horofunctions implies $\bg^{-1}\Par^r(h) \in \sH_{r+i}$ for some $i$ with $|i| \le d(\bg,e)$. Since $\Par^r(\psi_r(h) ) = \bg^{-1}\Par^r(h)$, it follows that $\psi_r(h) \in \sH_i$.
\end{proof}

\begin{defn}
Let $\sN^r_j \subset \sN^r$ be the set of those $h \in \sN^r$ such that $\psi_r(h) \in \sH_j$. The collection $\{\sN^r_j\}_{j\in\Z}$ partitions $\sN^r$. 
\end{defn}

The only part of the proof of theorem \ref{thm:step1} in which we use the key lemma is in the next corollary.

\begin{cor}\label{cor:key}
There exists a $c>0$ such that the following holds. Let $E \subset \sH_0\times X$ be $r$-determined for some $r>0$. If $\eta \times \mu\big(E \cap (\sN^r \times X)\big)>0$ then 
$$\frac{\eta \times \mu\big(E \cap (\sN^r_s \times X)\big)}{\eta\times \mu\big(E \cap (\sN^r \times X)\big)} >c~\textrm{ and }~
\frac{\eta \times \mu\big(E \cap (\sN^r_t\times X)\big)}{\eta\times \mu\big(E \cap (\sN^r\times X)\big)} >c$$
where $s\ne t$ are the integers in lemma \ref{lem:key}. Note $0 \le s,t \le Q$ by definition of $Q$.
\end{cor}

\begin{proof}
Observe that by integrating over $x \in X$, it suffices to prove that if $F \subset \sH_0$ is $r$-determined then 
$$\frac{\eta (F \cap \sN^r_s )}{\eta(F \cap \sN^r )} >c\textrm{
and }
\frac{\eta (F \cap \sN^r_t )}{\eta(F \cap \sN^r )} >c.$$

Because $F$ is $r$-determined, it is a (finite) disjoint union of cylinder sets of order $r$. Thus we may assume that $F=[h]_r$ for some $h \in \sH_0$.

Since $[h]_r \cap \sN^r \ne \emptyset$, $\Par^r(F)= \sH_r \cap B$. Also, $\Par^r(F \cap \sN^r)=\sH_r \cap B \cap \bg C$. If $h \in \sN^r_j$ for some $j \in \Z$ then $\Par^r(h) \in \bg\sH_{r+j}$. This is because $\bg^{-1}\Par^r(h) = \Par^r(\psi_r(h)) \in \sH_{j+r}$. So $\Par^r(F \cap \sN^r_s) = \sH_r \cap B \cap \bg C \cap \bg\sH_{s+r}$. Since $F$ is a cylinder set, the map $\Par^r: F \to \sH_r \cap B$ is injective.

Recall that from lemma \ref{lem:isomorphic} that $M_R(\sH_*)$ and $M_R(\sH_0)$ are canonically isomorphic under a map $\Res^{-1}: M_R(\sH_0) \to M_R(\sH_*)$ which is the inverse to the normalized restriction map. So,
\begin{eqnarray*}
\frac{\eta (F \cap \sN^r_s )}{\eta(F \cap \sN^r )} &=& \frac{\Res^{-1} \eta (F \cap \sN^r_s )}{\Res^{-1} \eta(F \cap \sN^r )}\\
 &=& \frac{\Res^{-1} \eta \big(\Par^r(F \cap \sN^r_s) \big)}{\Res^{-1} \eta\big(\Par^r(F \cap \sN^r) \big)}\\
&=& \frac{\Res^{-1} \eta(\sH_r \cap B \cap \bg C \cap \bg\sH_{s+r})}{\Res^{-1} \eta(\sH_r \cap B \cap \bg C)}\\
&=& \frac{\alpha^r\Res^{-1} \eta(\sH_0 \cap B \cap \bg C \cap \bg\sH_s)}{\alpha^r \Res^{-1} \eta(\sH_0 \cap B \cap \bg C)}\\
&=& \frac{\alpha^r\eta( B \cap \bg C \cap \bg\sH_s)}{\alpha^r  \eta( B \cap \bg C)}.
\end{eqnarray*}
A similar statement holds for $t$ in place of $s$. The result now follow from lemma \ref{lem:key} applied to $\alpha^r \eta=\eta$.
\end{proof}

\begin{defn}
For $h \in \sH$ let $\bh \in \sH_0$ be the horofunction $\bh(g)=h(g)-h(e)$. For $(h,x) \in \sH \times X$, let $\overline{(h,x)} = (\bh, x)$. 
\end{defn} 

\begin{defn}
Define $\bphi_r: \sN^r \times X \to \sH_0 \times X$ as follows. For $h \in \sN^r$, let $k_h \in G$ be such that $\psi_r(h)=k_hh$. Then define $\bphi_r(h,x)=(\overline{k_hh}, k_hx)$. 
\end{defn}

\begin{lem}\label{lem:bphi}
There exists a constant $C_0>0$ (that does not depend on $r$) such that for all $(h,x) \in \sN^r \times X$,
$$C_0^{-1} \le \frac{d\bphi_{r*}(\eta\times \mu)}{d(\eta\times \mu)}(h,x) \le C_0.$$
\end{lem}
\begin{proof}
Note that $\bphi_r(h,x) = (\phi(k_h)h, k_h x)$ where $\phi$ is as in definition \ref{def:phi} and $k_h$ is as in the previous definition. Now $h(e)=0$ and $h(k_h^{-1}) = -i $ iff $k_hh = \psi_r(h) \in \sH_i$. So, by definition \ref{defn:psi} item (3),  $|h(k_h^{-1})-h(e)| \le d(e,\bg)$. Lemma \ref{lem:precise} now implies the claim.
\end{proof}

\begin{defn}\label{defn:E'}
Let $\epsilon>0$ be such that 
$$\frac{\eta\times\mu(E_0)-(1+C_0)(1+Q)\epsilon}{\eta\times\mu(E_0)+(1+C_0)(1+Q)\epsilon}\ge 1-\frac{c}{2}$$
where $c>0$ is an in corollary \ref{cor:key} and $C_0$ is as in the previous lemma. By lemma \ref{lem:approx}, there exist sets $E'_i$ (for $i=0...Q$) such that
\begin{itemize}
\item for some $r>0$, for all $i$, $E'_i$ is $r$-determined,
\item $\eta\times \mu(E_i \Delta E'_i) < \epsilon$ for all $i$.
\end{itemize}
\end{defn}

\begin{defn} Let $E_{ij} = E_i \cap \bphi_r^{-1}(E_j)$. Let $E'_{ij} = E'_i \cap \bphi_r^{-1}(E'_j)$.
\end{defn}

\begin{lem}\label{lem:2epsilon} 
For any $0\le i,j \le Q$,
\begin{enumerate}
\item $E_{ij}, E'_{ij} \subset \sN^r \times X$,
\item Let $\sN^r_* = \sN^r - \cup_{0\le k \le Q} \, \sN^r_k$. Then $E_{0j} \subset (\sN^r_j \cup \sN^r_*)\times X$ (up to a measure zero set),
\item $\eta \times \mu(E_{ij} \Delta E'_{ij}) < (1+C_0)\epsilon$
\end{enumerate}
where $C_0>0$ is as in lemma \ref{lem:bphi}.
\end{lem}
\begin{proof}
\begin{enumerate}
\item This follows from Dom$(\psi_r) \subset \sN^r$.

\item If $(h,x) \in E_0$ and $g \in G$ is such that $g(h,x) \in \sH_k \times X$ for some $k$ with $0 \le k \le Q$ then by definition of $E_i$, it follows that $\overline{g(h,x)} \in E_k$ with probability one. Thus, if $\bphi_r(h,x)=\overline{g(h,x)}$ (for some $g \in G$) is in $E_j$ for some $0 \le j \le Q$ it follows from the pairwise disjointness of the sets $\{E_k\}_{k=0}^Q$ that $gh \in \sH_j$. This implies $(h,x) \in (\sN^r_j \cup \sN^r_*) \times X$ with probability one.

\item Observe that $E_{ij} \Delta E'_{ij} \subset E_i \Delta E'_i \cup \bphi_r^{-1}(E_j \Delta E'_j)$. Since $\eta\times \mu(E_i \Delta E'_i) < \epsilon$ and  $\eta\times \mu( \bphi_r^{-1}(E_j \Delta E'_j)) < C_0\epsilon$ (by lemma \ref{lem:bphi}) the lemma follows.
\end{enumerate}
\end{proof}

\begin{lem}\label{lem:previous}
For any $i,j$, let $E''_{ij}$ be the smallest $r$-determined set containing $E'_{ij}$. 
Then $E''_{ij} \cap (\sN^r \times X) = E'_{ij}$. Thus for any $k$, $E''_{ij} \cap (\sN^r_k \times X) = E'_{ij} \cap (\sN^r_k \times X)$.
\end{lem}

\begin{proof}
 It is immediate that $E'_{ij} \subset E''_{ij} \cap (\sN^r \times X)$. Let $(h,x) \in E''_{ij} \cap (\sN^r \times X)$. We need to show that $(h,x) \in E'_{ij}$. Because $E'_i$ is $r$-determined it is immediate that $(h,x) \in E'_i$. Therefore it suffices to show that $\bphi_r(h,x) \in E'_j$. By definition, there exists $h' \in \sH_0$ such that $(h',x) \in E'_{ij}$ and 
$$\Block(\Par^n(h))=\Block(\Par^n(h')) \textrm{ for all } 0\le n \le r.$$

Recall that $g_h \in G$ is such that $g_hh = \Par^r(h)$. It follows that $g_hh'=\Par^r(h')$ too. By definition, $\psi_r(h)=f_r(  g_h )^{-1}\bg^{-1}\Par^r(h)$. Hence $\psi_r(h')=f_r(  g_h )^{-1}\bg^{-1}\Par^r(h')$. Thus if $k_h \in G$ is such that $k_h h =\psi_r(h)$ then $k_h h'=\psi_r(h')$ too.

Since $E'_j$ is $r$-determined, it suffices to show that $\Block(\Par^n(\psi_r(h))) = \Block(\Par^n(\psi_r(h')))$ for all $0 \le n \le r$. Since both $h$ and $h'$ are in $\sN_r$, this is true for $n=r$. The general case follows from lemma \ref{lem:determined} and the fact that if $k_h \in G$ is such that $\psi_r(h) = k_hh$ then $\psi_r(h')=k_hh'$, too.
\end{proof}

\begin{proof}[Proof of theorem \ref{thm:step1}]
By lemma \ref{lem:2epsilon} item (2), $E_{0j} \subset (\sN^r_j \cup \sN^r_*) \times X$. By lemma \ref{lem:2epsilon} item (3) this implies that when $\eta \times \mu(E'_{0j})>0$,
$$\frac{\eta \times \mu(E'_{0j} \cap (\sN^r_j \cup \sN^r_*) \times X)}{\eta \times \mu(E'_{0j})} \ge \frac{\eta \times \mu(E_{0j}) - (1+C_0)\epsilon}{\eta \times \mu(E_{0j}) + (1+C_0)\epsilon}.$$
Since the collection $\{E_{0i}\}_{i=0}^Q$ partitions $E_0$, it follows that there is some $i\ge 0$ with $Q \ge i$ such that
$$\eta \times \mu(E_{0i}) \ge \frac{\eta \times \mu(E_0)}{1+Q}.$$
Fix this value of $i$. By lemma \ref{lem:2epsilon} item (3),
$$\eta \times \mu(E'_{0i}) \ge \eta\times\mu(E_{0i}) - (1+C_0)\epsilon \ge \frac{\eta\times\mu(E_0)}{1+Q}-(1+C_0)\epsilon>0.$$
The last inequality follows from the choice of $\epsilon$. So,
$$\frac{\eta \times \mu(E'_{0i} \cap (\sN^r_i \cup \sN^r_*) \times X)}{\eta \times \mu(E'_{0i})} \ge \frac{\eta \times \mu(E_{0}) - (1+C_0)(1+Q)\epsilon}{\eta \times \mu(E_{0}) + (1+C_0)(1+Q)\epsilon} \ge 1-\frac{c}{2}.$$

The second inequality is true by the choice of $\epsilon$ (see definition \ref{defn:E'}). On the other hand, corollary \ref{cor:key} implies that there are integers $s \ne t$ such that $0\le s,t \le Q$ and
$$\frac{\eta\times\mu\big(E''_{0i} \cap (\sN^r_s\times X)\big)}{\eta \times \mu\big(E''_{0i}\cap (\sN^r \times X)\big)} \ge c \textrm{ and } \frac{\eta\times\mu\big(E''_{0i} \cap (\sN^r_t\times X)\big)}{\eta \times \mu\big(E''_{0i}\cap (\sN^r \times X)\big)} \ge c.$$
Since $E''_{0i} \cap (\sN^r_s \times X)= E'_{0i}\cap (\sN^r_s\times X)$ (lemma \ref{lem:previous}) and $E'_{0i}=E'_{0i}\cap (\sN^r\times X)=E''_{0i} \cap (\sN^r \times X)$ (by definition), this implies that
$$\frac{\eta \times \mu\big(E'_{0i} \cap (\sN^r_s\times X)\big)}{\eta \times \mu(E'_{0i})} = \frac{\eta \times \mu\big(E''_{0i} \cap (\sN^r_s\times X)\big)}{\eta \times \mu\big(E''_{0i} \cap (\sN^r \times X)\big)} \ge c.$$
Similarly, 
$$\frac{\eta \times \mu\big(E'_{0i} \cap (\sN^r_t\times X)\big)}{\eta \times \mu(E'_{0i})} \ge c.$$
Since $s \ne t$, we may assume (after switching $s$ and $t$ if necessary) that $i\ne s$. Since the sets $\sN^r_j$ ($j \in \Z$) partition $\sN^r$, it follows that 
\begin{eqnarray*}
1&=&\frac{\eta \times \mu\big(E'_{0i} \cap (\sN^r\times X)\big)}{\eta \times \mu(E'_{0i})}\\
&\ge& \frac{\eta \times \mu\Big(E'_{0i} \cap \big((\sN_i^r \cup \sN^r_*) \times X\big)\Big)}{\eta \times \mu(E'_{0i})} + \frac{\eta \times \mu\big(E'_{0i} \cap (\sN_s^r\times X)\big)}{\eta \times \mu(E'_{0i})}\\
&\ge& 1-\frac{c}{2} + c > 1.
\end{eqnarray*}
 This contradiction implies the theorem.
\end{proof}

\section{Proof of the key lemma}\label{section:horospherical}


 If $A, B \subset G$ then let $AB=\{ ab \in G|\, a \in A, b \in B\}$. The goal of this section is to prove the key lemma \ref{lem:key}. However, we first prove the following helpful proposition.


\begin{prop}\label{horocycle theorem}
There exists a finite set $R \subset G$ such that for any $s \in \Z$ there exists a finite set $F(s) \subset G$ such that
$$\big\{ g \in G ~\big|~ \exists h \in \sH \textrm{ such that } h(g)=h(e)+s\big\}R \cup F(s) = G.$$
\end{prop} 

We will say that a constant or a function is {\bf universal} if it depends only on $(G,A)$. 


\begin{lem}
There exists a universal constant $K>0$ such that if $T>K$ and $r_1:[0,T] \to \Gamma$ is any geodesic then there exists a geodesic ray $r_2:[0, \infty) \to \Gamma$ such that 
$$r_1(t)=r_2(t) \, , \, \forall 0 \le t\le T-K.$$
\end{lem}

\begin{proof} A word in the generating set $A$ is called geodesic if the path in the Cayley graph $\Gamma$ that it determines (starting from the identity element) is geodesic. It is well-known that because $G$ is word hyperbolic, the set of all geodesic words forms a regular language $\sL$ (e.g., theorem 3.4.5 of [ECHLPT92]). Equivalently, $\sL$ is recognized by a deterministic finite state automaton. This fact is contained in theorem 1.2.7 of [ECHLPT92] where it is attributed to Kleene, Rabin and Scott. 

It is easy to see that if $\sL$ is any infinite language accepted by a finite state automaton then there exists a $K>0$ (depending only on $\sL$) such that if $w=s_1...s_n \in \sL$ and $n>K$ then the subword $s_1...s_{n-K}$ is infinitely extendable in the following sense. There exists an infinite word in $\sL$ that begins with $s_1...s_{n-K}...$. This is because there are only a finite number of states that inevitably lead to a failed state. So if $K$ is larger than the number of such states, the path $s_1...s_{n-K}$ must necessarily end in a state that does not inevitable lead to a failed state and is therefore, infinitely extendable.

Now let $r_1:[0,T] \to \Gamma$ be a geodesic. Then the word determined by $r_1$ restricted to $[0,T-K]$ is infinitely extendable. This implies the lemma.
\end{proof}

\begin{lem}
If $x,y \in G$ then there exists a bi-infinite geodesic $\gamma$ with $d(x,\gamma),d(y,\gamma) \le 2K$ where $K>0$ is as in the previous lemma. Here, $d(x,\gamma)=\inf\{d(x,y)\,|\, y \in \gamma\}$.
\end{lem}

\begin{proof}
If $d(x,y) \le 2K$ then let $\gamma$ be any bi-infinite geodesic through $x$. Otherwise, let $r_1:[0,T]\to \Gamma$ be a geodesic with $r_1(0)=x$, $r_1(T)=y$. It follows from the previous lemma that $r_1$ restricted to $[K,T-K]$ can be extended to a bi-infinite geodesic: there exists a geodesic $r_2:\R \to \Gamma$ such that $r_2(t)=r_1(t)$ for all $t \in [K,T-K]$. Let $\gamma$ be this geodesic. In both cases, $d(x,\gamma),d(y,\gamma) \le 2K$. 
\end{proof}
The proof of the next lemma is similar to the one above.
\begin{lem}
If $x, y \in G$ then there exists a geodesic ray $r:[0,\infty)\to X$ such that $r(0)=x$ and $d(y,r[0,\infty)) \le K$ where $K>0$ is as in the previous lemma.
\end{lem}

\begin{lem}[Existence of Coarse Perpendiculars]
There exists a universal function $N:[0,\infty) \to [0,\infty)$ satisfying the following. Let $\gamma$ be any geodesic, $z$ be any vertex on $\gamma$ and $C>0$. Then there exists a geodesic ray $r:[0,\infty) \to \Gamma$ and a number $t \in [0, N(C)]$ such that $r(0)=z$ and $d(r(t),\gamma) > C$.
\end{lem}

\begin{proof}
Recall that for $g \in G$, $S(g,n)$ and $B(g,n)$ denote the sphere and the ball of radius $n$ centered at $g \in G$. Let $\Nb(\gamma,K+C)$ denote the radius-$(K+C)$ neighborhood of $\gamma$. Then $\big|\Nb(\gamma,K+C) \cap S(z,n)\big|$ grows linearly in $n$ while $|S(z,n)|$ grows exponentially (by theorem \ref{thm:growth}). So there exists a universal function $N_1:\R\to \R$ such that for all $n\ge N_1(K+C)$, $S(z,n) \nsubseteq \Nb(\gamma,K+C)$.

Let $w \in S(z,N_1(K+C))$ be such that $d(w,\gamma)>K+C$. By the previous lemma, there exists a geodesic ray $r:[0,\infty)\to \Gamma$ such that $r(0)=z$ and $d(w, r[0,\infty)) \le K$. Thus there exists a $t>0$ with $d(w,r(t))\le K$. By the triangle inequality, this implies $d(\gamma,r(t)) \ge C$. 

By the triangle inequality again, $t=d(r(t),z)) \le d(r(t),w) + d(w,z) \le K + K+N_1(C)$. Hence the lemma is proven with $N(C)=2K + N_1(C)$.

\end{proof}

\begin{lem}[Existence of Coarse Isosceles Triangles]
There exists a universal constant $C>0$ such that the following holds. Let $\gamma \subset \Gamma$ be any geodesic. Let $z$ be a vertex on $\gamma$. Let $a$ and $b$ be the endpoints of $\gamma$ (which may be on $\partial \Gamma$). Then there exists a point $c \in \partial \Gamma$ so that if $[a,c]$ and $[b,c]$ are any two geodesics from $a$ to $c$ and from $b$ to $c$ respectively then there exist vertices $a' \in [a,c]$ and $b' \in [b,c]$ so that $d(a',z),d(b',z) \le C$.
\end{lem}

\begin{proof}
Recall that $\delta>0$ is the hyperbolicity constant of $\Gamma$. By the previous lemma, there exists a geodesic ray $r:[0,\infty)\to \Gamma$ and a $t\ge 0$, so that $r(0)=z$, $d(r(t), \gamma) \ge \delta$ and $t \le N(\delta)$. Let $c=r(\infty) \in \partial \Gamma$. Let $[a,c]$, $[b,c]$ be any geodesics from $a$ to $c$ and from $b$ to $c$ respectively. 

Consider the triangle with vertices $a,z,c$ and geodesic sides $\gamma$, $r$ and $[a,c]$. Because this triangle is $\delta$-thin, the point $r(t)$ is within the $\delta$-neighborhood of $\gamma \cup [a,c]$. By construction, $d(r(t),\gamma) > \delta$. Hence there exists an element $a' \in [a,c]$ with $d(r(t),a')\le \delta$. Similarly, there exists a vertex $b' \in [b,c]$ with $d(r(t),b') \le \delta$.

Since $t\le N(\delta)$, this implies that $d(z,a') \le \delta + N(\delta)$ and $d(z,b')\le \delta +N(\delta)$. This implies the lemma with $C=\delta + N(\delta)$. 

\end{proof}

\begin{lem}
There exists a universal constant $B$ such that if $x,y \in G$, $s \in \N$ and $d(x,y) >s \ge 0$ then there exists a horofunction $h \in \sH$ such that $|h(x)-h(y)-s|\le B$.
\end{lem}

\begin{proof}
Let $\gamma$ be a geodesic in $\Gamma$ from $x$ to $y$. Let $z' \in G$ be an approproximate midpoint of the geodesic segment between $x$ and $y$. That is to say, $z'$ is a vertex on $\gamma$ satisfying
$$d(z',x)\le \frac{d(x,y)+1}{2} \textrm{ and } d(z',y) \le \frac{d(x,y)+1}{2}.$$
Let $z$ be a vertex on $\gamma$ with $d(x,z)=d(x,z')+\lfloor s/2 \rfloor$.

By the previous lemma, there exists a point $c \in \partial \Gamma$ so that if $[x,c]$ and $[y,c]$ are any two geodesics from $x$ to $c$ and from $y$ to $c$ respectively then there exists vertices $x' \in [x,c]$ and $y' \in [y,c]$ so that $d(x',z),d(y',z) \le C$ where $C$ is universal.

Recall that the Busemann cocycle associated to a geodesic ray $r:[0,\infty) \to \Gamma$ is the function $\phi:G\times G \to \R$ defined by
$$\phi(g_1,g_2) = \lim_{t\to \infty} d(g_1,r(t)) - d(g_2,r(t)).$$
It satisfies the cocycle identity $\phi(g_1,g_3) = \phi(g_1,g_2)+\phi(g_2,g_3)$, the antisymmetry $\phi(g_1,g_2)=-\phi(g_2,g_1)$ and the inequality $|\phi(g_1,g_2)| \le d(g_1,g_2)$ for all $g_1,g_2,g_3 \in G$. The function $h(g):=\phi(g,e)$ is a horofunction [CP01, Proposition 2.9] and $\phi(g_1,g_2)=h(g_1)-h(g_2)$. 

Let $\phi_x$ be the Busemann cocycle associated to the ray $[x,c]$ and let $\phi_y$ be the Busemann cocycle associated to the ray $[y,c]$. Let $h_x \in \sH$ be the Busemann horofunction $h_x(g):=\phi_x(g,e)$. It suffices to show that $|h_x(x)-h_x(y)-s|\le B$, i.e., $|\phi_x(x,y) -s| \le B$ for some universal constant $B$.

By the cocycle identity, $\phi_x(x,y) = \phi_x(x,x') + \phi_x(x',y') + \phi_x(y',y)$. Since $|\phi_x(x',y')| \le d(x',y') \le 2C$, 
$$|\phi_x(x,y)-s| \le |\phi_x(x,x') + \phi_x(y',y)-s| + 2C.$$
By theorem \ref{thm:CP1}, $|\phi_x(y',y) - \phi_y(y',y)| \le 128\delta$. Thus,
$$|\phi_x(x,y)-s| \le |\phi_x(x,x') + \phi_y(y',y)-s| + 2C + 128\delta.$$
Since $x' \in [x,c]$, $\phi_x(x,x')= d(x,x')$. Since $y' \in [y,c]$, $\phi_y(y',y) = - d(y',y)$. Hence 
$$|\phi_x(x,y)-s| \le |d(x,x') - d(y,y')-s| + 2C + 128\delta.$$
According to the triangle inequality, $|d(x,x')- d(x,z)| \le d(z,x')\le C$ and $|d(y,y')- d(y,z)| \le d(z,y')\le C$. So
$$|\phi_x(x,y)-s| \le |d(x,z) - d(y,z)-s| + 4C + 128\delta.$$
By the definition of $z$, $d(x,z)=d(x,z')+\lfloor s/2 \rfloor$, $d(y,z)=d(y,z')-\lfloor s/2 \rfloor$. Thus
$$ |d(x,z) - d(y,z)-s| \le |d(x,z')-d(y,z')| + 1.$$
Since $z'$ is an approximate midpoint, $|d(x,z')-d(y,z')| \le 1$. This proves the lemma with $B=2+4C+128\delta$.
\end{proof}

\begin{proof}[Proof of proposition \ref{horocycle theorem}]
Let $R$ be the ball of radius $B$ centered at the identity in $G$. Let $F(s)$ be the ball of radius $|s+1|$ centered at the identity in $G$. The previous lemma implies
$$\big\{ g \in G ~\big|~ \exists h \in \sH \textrm{ s.t. } h(g)=h(e)+s\big\}R \cup F(s) = G.$$

\end{proof}


\begin{defn}\label{def:sBr}
A block $B \in \sB$ is {\bf recurrent} if it is contained in a directed cycle. Here we are considering $\sB$ as a directed graph (definition \ref{defn:tsB}). Let $\sB_r \subset \sB$ denote the set of recurrent blocks. Observe that there exists a constant $C_r>0$ such that if $h \in \sH_0$ is arbitrary then $\Block(\Par^n(h))$ is recurrent for all $n \ge C_r$. Indeed, we can choose $C_r$ to be the number of nonrecurrent blocks. From this it follows that for any $\eta \in M_R(\sH_0)$ there exists some $B \in \sB_r$ such that $\eta(B)>0$.
\end{defn}

\begin{lem}
If $h\in\sH$, $g_1\in G$, $q\in \Z$ and for some $C_1>0$, $|h(g_1)-q| \le C_1$ then there exists $g_2 \in G$ satisfying
\begin{itemize}
\item $d(g_2,g_1)\le C_1+ C_2$ where $C_2$ is a universal constant,
\item $h(g_2) = q$,
\item $\Block(g_2^{-1}h)$ is recurrent.
\end{itemize}
\end{lem}

\begin{proof}
Set $f=\Par_h^n(g_1)$ where $n \ge 0$ is chosen so that $\Block(f^{-1}h)$ is recurrent and $n \le C_r$ is universally bounded. 

Recurrence implies that $\Block(\Par^m_h(f)^{-1}h)$ is recurrent for all $m\ge 0$. So if $q \le h(f)$ then we may set $g_2 = \Par^m_h(f)$ where $m=h(f)-q$. In this case $d(g_2,g_1)\le |h(g_1)-q|\le C_1$. So we can set $C_2=0$.

Suppose that $q > h(f)$. Recurrence implies that for every $m\ge 0$ there exists $g_2 \in G$ such that $\Par^m_h(g_2)=f$ and $\Block(g_2^{-1}h)$ is recurrent. So set $m=q-h(f)$. Then $h(g_2)=q$ and 
$$d(g_2,g_1) \le d(g_2,f) + d(f,g_1) = m + n \le |q-h(f)|+C_r \le |h(g_1)-q| +2C_r.$$
So we are done.
\end{proof}

\begin{defn}
For $\eta \in M_R(\sH_0)$, $q \in \Z$ and $B,C \in \sB_r$, let 
$$G_\eta(q,B,C)=\{g \in G\, |  \, \eta(g\sH_q \cap B \cap gC)>0\}.$$
Let $G_\eta(q)=\cup_{B,C} \, G_\eta(q,B,C)$.
\end{defn}

\begin{lem}
There exist finite sets $L,R, F(q) \subset G$ satisfying the following. For all $\eta \in M_R(\sH_0)$ such that $\alpha\eta=\eta$,
$$LG_\eta(q)R \cup F(q) = G.$$
$L$ and $R$ do not depend on $\eta$ or $q$.
\end{lem}

\begin{proof}
Let $F(q)$ be as in proposition \ref{horocycle theorem}. Let $g_0 \in G-F(q)$ be arbitrary. The same proposition implies that there exists $h_1 \in \sH_0$ and $g_1\in G$ such that $d(g_0,g_1)\le C_0$ (where $C_0$ is a universal constant) and $h_1(g_1)=-q$.

 By lemma \ref{lem:quasiconformal}, $\pi_*(\eta)$ is quasiconformal. Theorem \ref{thm:ps} now implies that the support of $\pi_*(\eta)$ is all of $\partial \Gamma$. In particular, there exists $h_2 \in  \support(\eta)$ such that $\pi(h_2)=\pi(h_1)$. By theorem \ref{thm:CP1}, $||h_2-h_1||_\infty \le C_1$ where $C_1$ is a universal constant. Since $h_1(g_1)=-q$ the previous lemma implies that there exists $g_2 \in G$ such that
\begin{itemize}
\item $d(g_2,g_1) \le C_1+C_2$, where $C_2$ is a universal constant,
\item $h_2(g_2)=-q$ and
\item $\Block(g_2^{-1}h_2)$ is recurrent.
\end{itemize}
It follows from the previous lemma that there exists $e_1 \in G$ such that 
\begin{itemize}
\item $d(e,e_1) \le C_2$,
\item $h_2(e_1)=0$ and
\item $\Block(e_1^{-1}h_2)$ is recurrent.
\end{itemize}
Since $h_2(e_1)=0$ it follows that $h_3:=e_1^{-1}h_2$ is in the support of $\eta$. Set $g_3:=e_1^{-1}g_2$. Observe that:
\begin{itemize}
\item $h_3 \in \support(\eta)$,
\item $\Block(h_3)$ is recurrent,
\item $\Block(g_3^{-1}h_3) = \Block(g_2^{-1}h_2)$ is recurrent,
\item $h_3(g_3)=h_2(g_2)=-q$ and
\item there exist elements $l,r \in G$ such that $lg_3r=g_0$ and $d(l,e),d(r,e) \le C_3$ where $C_3 \ge 0$ is a universal constant.
\end{itemize}
Consider the set $g_3 \sH_q \cap \Block(h_3) \cap g_3\Block(g_3^{-1}h_3)$. It contains $h_3$ and so, it has nontrivial intersection with $\support(\eta)$. Since it is clopen, this implies $\eta(g_3 \sH_q \cap \Block(h_3) \cap g_3\Block(h_3))>0$. Thus $g_3 \in G_\eta(q)$. Since $lg_3r=g_0$ and $g_0$ is arbitrary, this implies the lemma: $L$ and $R$ are the set of all elements in $G$ with distance at most $C_3$ from the identity element.

\end{proof}



\begin{proof}[Proof of lemma \ref{lem:key}]
 Let $N>0$ be larger than the product $|L||R||\sB_r|^2$ where $L,R$ are as in the previous lemma and $\sB_r$ is as in definition \ref{def:sBr}. Let $f \in G - \cup_{q=0}^N F(q)$. By the previous lemma $LG_\eta(q)R \cup F(q) = G$ for all $q$. Therefore, for every $q$ with $0\le q \le N$, there exists elements $l_q \in L^{-1}$, $r_q \in R^{-1}$ and $B_q, C_q \in \sB_r$  such that $l_qfr_q \in G_\eta(q,B_q,C_q)$. 

By the pigeonhole principle, there must exist integers $s,t$ such that $0 \le s <t\le N$ such that $l_s=l_t$, $r_s=r_t$, $B_s=B_t$ and $C_s=C_t$. Let $\bg= l_s f r_s=l_t f r_t$. Then $\bg \in G_\eta(s,B_s,C_s) \cap G_\eta(t,B_t,C_t)$ which implies the lemma.

\end{proof}

\section{Proof of theorem \ref{thm:app2} and corollary \ref{cor:app}}\label{sec:app}

\begin{proof}[Proof of theorem \ref{thm:app2}]

We start with a number of reductions. Since $X$ is finite and $G \curvearrowright (X,\mu)$ is ergodic, it follows that $X=G/H$ for some finite index subgroup $H<G$ and $\mu$ is the uniform measure. The first step is to reduce to the case in which $H$ is normal.

Suppose that the theorem is true whenever $X=G/N$ where $N$ is a finite-index normal subgroup of $G$. Let $H<G$ be an arbitrary finite index subgroup. Then there exists a finite index normal subgroup $N<G$ such that $N<H$ (for example, let $N$ be the intersection of all the conjugates of $H$ in $G$). Since $G/N$ projects onto $G/H$ in a $G$-equivariant manner, if the theorem is true when $X=G/N$, it must be true when $X=G/H$. So it suffices to assume that $X=G/N$ where $N\vartriangleleft G$.

$G$ admits an action on $\Z^G \times G/N$ {\it on the right} by 
$$\Phi(g)(h,xN) = (h, xNg) = (h,xgN).$$
This action commutes with the diagonal left action (i.e., the action $g(h,xN) = (gh,gxN)$ on $\Z^G \times G/N$). Hence, it pushes forward to an action on $M_R(\sH_0 \times G/N)$ as well as $M(\Z^G \times G/N)$. This action is weak* continuous and linear. 

Suppose that the theorem is true whenever $x=N \in G/N$, i.e., every subsequential limit point of the sequence
$$ \Big\{\frac{1}{|K|}\sum_{k \in K} u_{n,kN}\Big\}$$ is of the form $\eta \times \mu$ for some probability measure $\eta \in M(\Z^G)$. For $g \in G$, $\Phi(g)_*$ is weak* continuous. So the above implies that every subsequential limit point of the sequence
$$ \Big\{\frac{1}{|K|}\sum_{k \in K} \Phi(g_*)u_{n,kN}\Big\} = \Big\{\frac{1}{|K|}\sum_{k \in K} u_{n,kgN}\Big\} $$ 
is of the form $\Phi(g)_*(\eta \times \mu)=\eta\times \mu$ for some probability measure $\eta \in M(\Z^G)$. So the theorem is true for $gN \in G/N$ too. Since $g \in G$ is arbitrary, it now suffices to show that the theorem is true when $x=N\in G/N$.

Let $\omega \in M_R(\sH_0 \times G/N)$. The measure 
$${\bar \omega} := \frac{1}{|G/N|} \sum_{gN \in G/N} \Phi(g)_*\omega$$
is a product measure ${\bar \omega} = \eta \times \mu$ for some $\eta \in M_R(\sH_0)$. So $\omega$ is absolutely continuous to $\eta \times \mu$ for some $\eta \in M_R(\sH_0)$. 

(Incidentally, since $\omega$ is arbitrary, theorem \ref{thm1} and theorem \ref{thm:main} now imply that there is a bound on the number of ergodic measures in $M_R(\sH_0 \times G/N)$ that depends only on $(G,A)$ and not on $N$.)

Let $\eta \in \sH_0$ be ergodic. By theorem \ref{thm:main}, $\eta \times \mu$ splits into a finite number of ergodic components: $\eta \times \mu = t_1\omega_1 + ... + t_q \omega_q$ where $q \le Q=Q(G,A)$.

 Let $G_0<G$ be the intersection of all subgroups of index at most $Q!$. Since $Q$ depends only on $(G,A)$, $G_0$ does not depend on $N$. Since $\Phi(g)_* (\eta \times \mu) = \eta \times \mu$, the $\Phi_*$-action of $G$ permutes the ergodic components of $\eta\times \mu$. Thus $\Phi_*$ gives a homomorphism from $G$ into the permutation group on $\{\omega_1,....,\omega_q\}$. Let $H$ denote the kernel of this action. Since $|G/H| \le q!$, it follows that $G_0 < H$. Since every $\omega \in M_R(\sH_0\times X)$ is absolutely continuous to a product measure, it follows that for all $g \in G_0$, $\Phi(g)_* \omega = \omega$. Thus
$${\bar \omega} := \frac{1}{|K|} \sum_{k \in K} \Phi(k)_*\omega = \eta \times \mu$$
for some $\eta \in M_R(\sH_0)$.

Because the $\Phi$-action of $G$ is weak* continuous, it follows that any subsequential limit of
$$\Big\{\frac{1}{|K|} \sum_{k \in K} \Phi(k)_* u_{n,N}\Big\} = \Big\{\frac{1}{|K|} \sum_{k \in K}  u_{n,kN}\Big\}$$
is of the form $\eta \times \mu$ for some $\eta \in M_R(\sH_0)$. Here we have used the fact (discussed in lemma \ref{lem:existence}) that every subsequential limit of $u_n$ is contained in $M_R(\sH_0)$. This concludes the proof.
\end{proof}









\begin{proof}[Proof of corollary \ref{cor:app}]
Let $f:{\bar G} \to \R$ be continuous. Let $G=H_0>H_1>...>e$ be a descending sequence of finite index subgroups of $G$ such that if $N<G$ is any finite index subgroup then $H_i < N$ for some $i$. For example, we could choose $H_i$ to be the intersection of all subgroups of $G$ of index at most $i$.

Let $f_i: {\bar G} \to \R$ be the conditional expectation of $f$ with respect to the $\sigma$-algebra on ${\bar G}$ obtained from the projection ${\bar G} \to G/H_i$ by pulling back the power set of $G/H_i$. By the martingale convergence theorem $\{f_i\}_{i=1}^\infty$ converges in the sup-norm to $f$ as $i\to\infty$.

By theorem \ref{thm:app}, for all $i\ge 0$, 
\begin{eqnarray*}
\int f_i d\mu &=& \lim_{n \to \infty} \frac{1}{|K||S(e,n)|} \sum_{g \in S(e,n)}  \sum_{k \in K} f_i(gkx).
\end{eqnarray*}

Let $\epsilon>0$. Then there exists $N$ such that for all $i>N$, $||f-f_i||_\infty < \epsilon$ and $|\int f_id\mu - \int f d\mu| <\epsilon$. The above equation now implies 
\begin{eqnarray*}
\Big|\int f d\mu -\lim_{n\to \infty} \frac{1}{|K||S(e,n)|} \sum_{g \in S(e,n)}  \sum_{k \in K} f(gkx) \Big| \le 2\epsilon.
\end{eqnarray*}
Since $\epsilon>0$ is arbitrary, this implies the corollary.
\end{proof}

\section{Conclusion}\label{sec:concl}



Given a discrete group $G$, let  $2^G$ denote the set of all subsets of $G$ with the product topology. It is a compact metrizable space (in fact, homeomorphic to a Cantor set). Let $2^G_e \subset 2^G$ denote the set of subsets of $G$ that contain the identity element $e \in G$. $G$ acts on $2^G$ in the usual way: if $S \subset G$ and $g \in G$ then $gS=\{gs |\, s \in S\}$. This action induces an equivalence relation on $2^G$. Let $R$ denote the restriction of this relation to $2^G_e$. 

Let $\sS:=M_R(2^G_e)$ denote the space of $R$-invariant Borel probability measures on $2^G_e$. This space generalizes the set of subgroups of $G$: if $H$ is a subset of $G$, then the Dirac measure $\delta_H$ concentrated at $H$ is in $\sS$ iff $H$ is a subgroup of $G$. It is interesting to think of measures in $\sS$ as being ``like subgroups''. For example, if $G \curvearrowright (X,\mu)$ and $\eta \in \sS$ then the ``induced action'' of $\eta$ on $X$ is the measure space $(2^G_e \times X, \eta \times \mu)$ with the equivalence relation induced by the diagonal action of $G$ on $2^G\times X$.

For another example, recall that a {\bf horosphere} of a word hyperbolic group $G$ is a level set of a horofunction. If $\eta \in \sS$ is concentrated on the space $HS$ of horospheres that contain the identity element, then it is interesting to speculate that the relationship between $\eta$ and $G$ should be analogous to the relationship between a maximal unipotent subgroup of $SO(n,1)$ and $SO(n,1)$. For example, it can be shown that the leaves of $(HS,\eta)$ have polynomial growth with respect to a very natural leafwise metric (cf. [Ad94]). The obvious map from $\sH_0$ to $HS$ is uniformly finite-to-1 (by theorem \ref{thm:CP2}) and relation-preserving. Thus theorems \ref{thm1} and \ref{thm:main} apply to $M_R(HS)$ in place of $M_R(\sH_0)$. 

For a third example, if $G$ is a 1-ended word hyperbolic group then, a well-known question (attributed to Gromov) asks, does $G$ have a subgroup $H$ isomorphic to the fundamental group of a closed surface of genus at least 2? A slight variation asks, does $G$ have a {\it quasiconvex} subgroup $H$ isomorphic to the fundamental group of a closed surface of genus at least 2? We can weaken this question to: is there a measure $\eta \in \sS$ such that each subset $S \in \support(\eta)$ is quasi-isometric to the hyperbolic plane?

\end{document}